\documentclass{article}
\usepackage{amsmath,amssymb}
\usepackage{graphicx,xypic}

\newtheorem{theorem}{Theorem}[section]
\newtheorem{lemma}[theorem]{Lemma}
\newtheorem{proposition}[theorem]{Proposition}
\newtheorem{corollary}[theorem]{Corollary}
\newtheorem{conjecture}[theorem]{Conjecture}
\newtheorem{definition}[theorem]{Definition}
\newtheorem{remark}[theorem]{Remark}
\newtheorem{example}[theorem]{Example}
\newtheorem{question}[theorem]{Questions}

\newcommand{\id}{\mbox{id}}

\newcommand{\Sym}{\operatorname{Sym}}

\newcommand{\FaM}{\operatorname{FaM}}

\newcommand{\mpl}{\operatorname{mpl}}
\newcommand{\Ret}{\operatorname{Ret}}
\newcommand{\Aut}{\operatorname{Aut}}
\newcommand{\Soc}{\operatorname{Soc}}
\newcommand{\Ker}{\operatorname{Ker}}
\newcommand{\Z}{\operatorname{Z}}

\newenvironment{proof}{\par\noindent{\bf Proof.}}{$\qed$\par\bigskip}
\newcommand{\qed}{\enspace\vrule  height6pt  width4pt  depth2pt}
\begin{document}

\title{Braces and the Yang-Baxter equation \thanks{ Research  partially
supported by grants of DGI MICIIN (Spain) MTM2011-28992-C02-01,
Generalitat de Catalunya 2009 SGR 1389, Onderzoeksraad of Vrije
Universiteit Brussel, Fonds voor Wetenschappelijk Onderzoek
(Belgium) and  MNiSW research grant N201 420539 (Poland).}}

\author{Ferran Ced\'o \and Eric Jespers \and Jan Okni\'nski}
\date{}
\maketitle

\begin{abstract}
Several aspects of relations between braces and non-degenerate
involutive set-theoretic solutions of the Yang-Baxter equation are
discussed and many consequences are derived. In particular, for each
positive integer $n$ a finite square-free multipermutation solution
of the Yang-Baxter equation with multipermutation level $n$ and an
abelian involutive Yang-Baxter group is constructed. This answers a
problem of Gateva-Ivanova and Cameron. It is also proved that finite
non-degenerate involutive set-theoretic solutions of the Yang-Baxter
equation whose associated involutive Yang-Baxter group is abelian
are retractable in the sense of Etingof, Schedler and Soloviev.
Earlier the authors proved this with the additional square-free
hypothesis on the solutions. Retractability of solutions is also
proved for finite square-free non-degenerate involutive
set-theoretic solutions associated to a left brace.
\end{abstract}

\section{Introduction}

The quantum Yang-Baxter equation appeared in a paper on statistical
mechanics by Yang \cite{Yang}.  It is one of the basic equations in
mathematical physics and it laid foundations of the theory of
quantum groups. Through the latter theory, the quantum Yang-Baxter
equation has also played an important role in the development of
Hopf algebras \cite{EtingofGelaki, radford}. One of the fundamental
open problems is to find all the solutions of the quantum
Yang-Baxter equation. Recall that a solution of the Yang-Baxter
equation is a linear map $R:V\otimes V \rightarrow V\otimes V$,
where $V$ is a vector space, such that
    $$R_{12}R_{13}R_{23} = R_{23}R_{13}R_{12},$$
where $R_{ij}$ denotes the map $V\otimes V \otimes V \rightarrow
V\otimes V \otimes V$ acting as $R$ on the $(i,j)$ tensor factor
and as the identity on the remaining factor. Since this problem
was stated, many solutions have been found and the related
algebraic structures have been intensively studied (see for
example \cite{kassel}). The simplest solutions are the solutions
$R$ induced by a linear extension of a mapping
$\mathcal{R}:X\times X \rightarrow X\times X$, where $X$ is a
basis for $V$.   In this case, one says that $\mathcal{R}$ is a
set-theoretic solution of the quantum Yang-Baxter equation.
Drinfeld, in \cite{drinfeld}, posed the question of finding these
set-theoretic solutions.

It is not difficult to see that if $\tau:X^{2}\rightarrow X^{2}$ is
the map defined by $\tau (x,y)=(y,x)$, then a map $\mathcal{R} :
X\times X \rightarrow X\times X$ is a set-theoretic solution of the
quantum Yang-Baxter equation if and only if the mapping $r=\tau
\circ \mathcal{R}$ is a solution of the braided equation (or a
solution of the Yang-Baxter equation, in the terminology used for
example in \cite{Gat,gat-maj})
    $$r_{12}r_{23}r_{12} = r_{23}r_{12}r_{23}.$$
Recently many results on set-theoretic solutions appeared in several
papers by Etingof, Schedler and Soloviev~\cite{ESS},
Gateva-Ivanova~\cite{Gat}, Gateva-Ivanova and Van den
Bergh~\cite{GIVdB}, Lu, Yan and Zhu~\cite{LYZ},
Rump~\cite{rump1,rump3,rump6}, Jespers and
Okni\'{n}ski~\cite{JO,JObook}, Ced\'o, Jespers and
Okni\'{n}ski~\cite{CJO}, Ced\'o, Jespers and del R\'{\i}o~\cite{CJR},
Gateva-Ivanova and Cameron~\cite{GC}, Gateva-Ivanova and
Majid~\cite{gat-maj} and others.  Recall that a bijective map
    $$\begin{array}{cccc} r\colon & X\times
    X&\longrightarrow &X\times X\\
    &(x,y)&\mapsto&(\sigma_x(y), \gamma_y(x)) \end{array}$$
is said to be involutive if $r^2=\id_{X^2}$. Moreover, it is said to
be left (respectively, right) non-degenerate if each map
$\gamma_{x}$ (respectively, $\sigma_{x}$) is bijective. Note that,
if $X$ is finite, then an involutive solution of the braided
equation is right non-degenerate if and only if it is left
non-degenerate (see \cite[Corollary~2.3]{JO} and \cite[Corollary
8.2.4]{JObook}). In \cite{ESS} and \cite{GIVdB} a beautiful group
theory translation was given of  the non-degenerate involutive
set-theoretic solutions of the Yang-Baxter equation, associating a
group,  denoted $G(X,r)$, to the solution $(X,r)$. If $X$ is finite,
then the groups obtained are called groups of $I$-type (see
Section~\ref{YBequation}) and are generated by $X=\{ x_{1},\ldots ,
x_{n}\}$ subject to defining relations $x_{i}x_{j}=x_{k}x_{l}$ (if
$r(x_{i},x_{j})=(x_{k},x_{l})$).  This approach also leads to a
class of permutation groups (referred to as involutive Yang-Baxter
groups when $X$ is finite \cite{CJR}), denoted by
$\mathcal{G}(X,r)$, that are generated by the corresponding
bijections $\sigma_{x}, x\in X$. It is hence important to
investigate the structure of these  two classes of groups in order
to classify the non-degenerate involutive set-theoretic solutions of
the Yang-Baxter equation. A possible strategy is outlined in
\cite{CJR}. Fundamental contributions in this direction have been
obtained in for example \cite{CJO,CJR,ESS,Gat,GC,rump1} and several
important open questions have been posed in
\cite{CJR,ESS,Gat-Mis,GC}.

Another approach, based on a new notion of a (right) brace, was
initiated by Rump in \cite{rump3}. This is a ring such that the
multiplication is right distributive with respect to the sum,
however it is not necessarily left distributive or associative, but
it satisfies some equation which generalizes both these conditions
(see Definition~\ref{def1}). The main reason for introducing this
algebraic object is that it allows another possible strategy to
attack the problem of classifying the non-degenerate involutive
set-theoretic solutions of the Yang-Baxter equation.

The main aim of this paper is to answer some of the fundamental
problems on non-degenerate involutive set-theoretic solutions of the
Yang-Baxter equation. First, we improve a result of \cite{CJO} on
such solutions $(X,r)$ for which the associated involutive
Yang-Baxter group is abelian and  the set $X$ is finite (see
Theorem~\ref{abelian}). Next, we answer two questions posed by
Gateva-Ivanova and Cameron in \cite{GC} (see
Theorem~\ref{elementary-abelian})  on the so called multipermutation
level, which is a certain measure of complexity of set-theoretic
solutions. Finally, we confirm a conjecture of Gateva-Ivanova made
in \cite{Gat} on set-theoretic solutions of the Yang-Baxter equation
determined by braces (see Theorem~\ref{square}). The proofs make use
of the delicate links made in the two possible strategies mentioned
above. Since the literature is quite scattered and because of
completeness and readability sake, we will give a brief and to the
point survey of the most relevant connections useful for this work.
This is done in a way that makes these connections direct and clear.
For the known results we give exact   references. We begin with a
section on braces where an explicit reformulation of the original
definition of a brace (Definition~\ref{defbrace}) is given in terms
of two related group structures on a given set. This viewpoint is
then used throughout the paper. We end with several examples of new
constructions of braces and  an exciting new reformulation by Sysak
\cite{sysak2} of the problems in the context of integral groups
rings, leading to strong links with the fundamental isomorphism
problem (asking whether a group is determined by its integral group
ring~\cite[Chapter~9]{PolcinoMiliesSehgal}). This again possibly
opens a new avenue for tackling some of the most challenging
problems on the solutions of the Yang-Baxter equation.

The outline of the remainder of the paper is as follows.\\
\begin{tabular}{l}
 Section~\ref{definitions}: Braces, homomorphisms and ideals.\\
 Section~\ref{YBequation}: Solutions of the Yang-Baxter
equation and related groups.\\
 Section~\ref{relationsYB}: Relations between the Yang-Baxter equation and braces.\\
 Section~\ref{groups}: Groups of $IG$-type and braces.\\
 Section~\ref{soclemulti}: Socle of a brace and multipermutation solutions.\\
  Section~\ref{openproblems}: About some open problems.\\
 Section~\ref{groupringrelations}: Relations between braces and group rings.\\
 Section~\ref{Section-Examples}: Constructions and examples of
 braces.
\end{tabular}

\section{Braces, homomorphisms and ideals}\label{definitions}
In \cite{rump3} Rump introduced braces as a generalization of
radical rings related with  non-degenerate involutive set-theoretic
solutions of the Yang-Baxter equation (see
Section~\ref{YBequation}). In his subsequent
papers~\cite{rump2,rump4,rump5,rump6,rump7}, he initiated the
development of the theory of this new structure. In this section we
recall its definition, give an equivalent definition and introduce
some basic concepts following Rump's papers, but  using this new
definition of brace.

First recall the definition of brace.
\begin{definition}\label{def1}  (Rump\cite[Definition
2]{rump3}) A right brace is an abelian group $(A,+)$  on which a
multiplication is defined such that the following properties hold
\begin{itemize}
\item[{(i)}] $(a+b)c=ac+bc$, for all $a,b,c\in A$,
\item[{(ii)}] $a(bc+b+c)=(ab)c+ab+ac$, for all $a,b,c\in A$,
\item[{(iii)}] for each $a\in A$, the  map $ A\longrightarrow A$, defined by
$x\mapsto xa+x$, is bijective.
\end{itemize}
\end{definition}

Let $A$ be a right brace. The circle operation $\circ$ on $A$ is
defined by
\begin{eqnarray}\label{circle}
a\circ b=ab+a+b,
\end{eqnarray}
for $a,b\in A$.

Let $A$ be an (additive) abelian group on which a multiplication is
defined that is right distributive. In \cite[Proposition 4]{rump3}
it is proved that $A$  is a right brace if and only if $A$ is a
group with respect to the circle operation~(\ref{circle}). For
$a,b,c\in A$,  $(a+b)c=ac+bc$ is equivalent to
$$(a+b)\circ c +c=a\circ c+ b\circ c.$$
Thus we reformulate the definition of a right brace as follows.

\begin{definition}\label{defbrace}
A right brace is a set $G$ with two operations $+$ and $\cdot$ such
that $(G,+)$ is an abelian group, $(G,\cdot )$ is a group  and
\begin{equation}\label{RE} (a+b)c+c=ac+bc,\end{equation}
 for all $a,b,c\in G$. We call $(G,+)$ the additive group and $(G,\cdot )$
 the multiplicative group
 of the right brace.
\end{definition}

A left brace is defined similarly, replacing condition (\ref{RE}) by
\begin{equation}\label{LE}
a(b+c)+a=ab+ac.\end{equation}

\begin{remark}
{\rm In what follows we will use Definition~\ref{defbrace} as the
definition of a right brace.}
\end{remark}

It is easy to check that in a right (left) brace $G$, the
multiplicative identity $1$ of the multiplicative group of $G$ is
equal to the neutral element $0$ of the additive  group of $G$.

\begin{definition}
Let $G$ be a right brace. The opposite brace of $G$ is the left
brace $G^{op}$ with the same additive group as $G$ and
multiplicative group equal to the opposite group of the
multiplicative group of $G$. The opposite brace of a left brace is
defined similarly.
\end{definition}

Thus there is a bijective correspondence between right braces
and left braces.

\begin{definition}
A two-sided brace is a right brace $G$ that  also is a left brace,
in other words, a right brace $G$ such that
\begin{eqnarray*} a(b+c)+a&=&ab+ac,\end{eqnarray*}
 for all $a,b,c\in G$.
\end{definition}

The next result is an easy consequence of the definition of brace,
it gives a characterization of two-sided braces and provides many
examples.

\begin{proposition}   \label{rad}
If $(G,+,\cdot )$ is a two-sided brace then $(G,+,* )$ is a radical
ring, where $*$ is the operation on $G$ defined by $a*b=ab-a-b$ for
$a,b\in G$. Conversely, if $(R,+,\cdot )$ is a radical ring then
$(R,+,\circ )$ is a two-sided brace,  where  $a\circ b = ab+a+b$,
for $a,b\in R$.
\end{proposition}

\begin{definition} \label{hom}  Let
$G$ and $H$ be two right (left) braces. A map $f\colon
G\longrightarrow H$ is a homomorphism of right (left) braces if
$f(a+b)=f(a)+f(b)$ and $f(ab)=f(a)f(b)$,  for all $a,b\in G$.
 The kernel of $f$ is $\Ker(f)=\{ a\in G\mid f(a)=1 \}$.
\end{definition}

Note that $\Ker(f)$ is a normal subgroup of the  multiplicative
group of a right (left) brace $G$. Since $1$ is also the neutral
element of the additive group of the brace $H$, we have that
$\Ker(f)$ also is a subgroup of the additive group of the brace $G$.
Thus $\Ker(f)$ is a right (left) brace.

For $a\in G$, we define $\rho_a, \lambda_a\in \Sym_G$ (the
symmetric group on $G$) by $\rho_a(b)=ba-a$ and
$\lambda_a(b)=ab-a$.

Note that if $b\in \Ker(f)$, then $\rho_a(b)\in\Ker(f)$ for
all $a\in G$.

\begin{definition} \label{ideal}
Let $G$ be a right (left) brace. An ideal of $G$ is a normal
subgroup $I$ of the multiplicative group of $G$ such that
$\rho_a(b)\in I,$ for all $a\in G$ and $b\in I$.
\end{definition}

Note that $\rho_a(b)=ba-a=aa^{-1}ba-a=\lambda_a(a^{-1}ba)$. Hence, a
normal subgroup $I$ of the  multiplicative group of a right (left)
brace $G$ is an ideal of $G$ if and only if $\lambda_a(b)\in I,$ for
all $a\in G$ and $b\in I$.

Let $I$ be an ideal of a right brace $G$. Let $a,b\in I$. Note that
$a-b=ab^{-1}b-b=\rho_b(ab^{-1})\in I$. Therefore $I$ is a right
brace. Let $G/I$ be the quotient group of the multiplicative group
of $G$ modulo its normal subgroup $I$. Let $c\in G$ and $b\in I$.
Note that $bc=\rho_c(b)+c\in I+c$ and hence also
$c+b=\rho_{c}^{-1}(b)c\in Ic$. Thus $Ic=I+c$ and therefore $G/I$
also is the quotient group of the additive group of $G$ modulo the
additive subgroup $I$.  It is easy to see that $(G/I,+,\cdot)$ is a
right brace. This is called the quotient brace of $G$ modulo $I$.
The quotient brace of a left brace modulo an ideal is defined
similarly.

Note that we have a natural short exact sequence of right braces
$$1\longrightarrow I\longrightarrow G\longrightarrow G/I\longrightarrow 1.$$
We say that $G$ is an extension of the brace $I$ by the brace $G/I$.

We often will make  use of the following easy consequence of
(\ref{LE}) for a left brace $G$:
\begin{eqnarray*}
c(a-b)-c  &=& ca-cb,
\end{eqnarray*}
for all $a,b,c\in G$. We now state some elementary  properties of
the maps $\lambda_a$.  The following two Lemmas are proved in
\cite{rump3} (see Proposition 2, Remark 1 and Proposition 5 in
\cite{rump3}).

\begin{lemma}\label{lambda}
Let $G$ be a left brace.
\begin{itemize}
\item[(i)] $\lambda_a(x+y)=\lambda_a(x)+\lambda_a(y)$, that is
$\lambda_ a$ is an automorphism of the abelian group $(G,+)$.
\item[(ii)] $\lambda_a\lambda_b=\lambda_{ab}$, that is the map
$\lambda\colon G\longrightarrow \Sym_G$, defined by
$\lambda(a)=\lambda_a$ is a homomorphism of groups.
\end{itemize}
\end{lemma}

\begin{proof}
$(i)$ Let $a,x,y\in G$. Then
\begin{eqnarray*}
\lambda_a(x+y)&=&a(x+y)-a=ax+ay-a-a=\lambda_a(x)+\lambda_a(y).
\end{eqnarray*}
$(ii)$ Let $a,b,x\in G$. Then
\begin{eqnarray*}
\lambda_a\lambda_b(x)&=&\lambda_a(bx-b)=a(bx-b)-a=abx-ab=\lambda_{ab}(x).
\end{eqnarray*}
\end{proof}

Similarly we prove the following result.

\begin{lemma}\label{rho}
Let $G$ be a right brace.
\begin{itemize}
\item[(i)] $\rho_a(x+y)=\rho_a(x)+\rho_a(y)$, that is
$\rho_ a$ is an automorphism of the abelian group $(G,+)$.
\item[(ii)] $\rho_a\rho_b=\rho_{ba}$, that is the map
$\rho\colon G\longrightarrow \Sym_G$, defined by $\rho(a)=\rho_a$ is
an antihomomorphism of groups.
\end{itemize}
\end{lemma}

In the context of Proposition~\ref{rad} we now show that  ideals of a
two-sided brace coincide with  the ideals of the corresponding  radical ring.

Let $R$ be a radical ring. Let $I$ be an ideal of the ring $R$. We
shall see that $I$ is an ideal of the two-sided  brace
$(R,+,\circ)$. Let $x\in I$ and $a,a'\in R$ be elements such that
$a\circ a'=a'\circ a=0$, that is $a'$ is the inverse of $a$ in  the
group $(R, \circ)$. We have
\begin{eqnarray*} a\circ x\circ
a'&=&(ax+a+x)a'+ax+a+x+a'\\
&=&axa'+aa'+xa'+ax+a+x+a'\\
&=&axa'+xa'+ax+x+a\circ a'=axa'+xa'+ax+x\in I.
\end{eqnarray*}
Hence $I$ is a normal subgroup of $(R,\circ )$. Furthermore
$$\rho_a(x)=x\circ a-a=xa+x+a-a=xa+x\in I.$$
Thus $I$ is an ideal of the brace $R$.

Conversely, let $J$ be an ideal of the brace $R$. We know that $J$
is a subbrace of $R$. Thus in order to prove that $J$ is an ideal of
the ring $R$ it is enough to show that $ax,xa\in J$ for all $x\in
J$ and all $a\in R$. Let $x\in J$ and $a\in R$. Let $a'\in R$ be the
inverse of $a$ in the  group $(R,\circ)$. We have
$$ax=a\circ x-a-x=\rho_a(a\circ x\circ a')-x\in J$$
and
$$xa=x\circ a-a-x=\rho_a(x)-x\in J.$$
Therefore $J$ indeed is an ideal of the ring $R$.

Note also that if $I$ is an ideal of the radical ring $R$, then
the ring $R/I$ is the radical ring corresponding to the quotient
brace of the two-sided brace $R$ modulo $I$.

\section{Solutions of the Yang-Baxter
equation and related groups}\label{YBequation}

Let $X$ be a non-empty set. Let $r\colon X^2\longrightarrow X^2$ be
a map and write  $r(x,y)=(\sigma_x(y), \gamma_y(x))$. We say that
$(X,r)$ is a right non-degenerate involutive set-theoretic solution
of the Yang-Baxter equation if
\begin{itemize}
\item[(1)] $r^2=\id_{X^2}$,
\item[(2)] $\sigma_x\in \Sym_X$, for all $x\in X$,
\item[(3)] $r_1r_2r_1=r_2r_1r_2$, where $r_1=r\times \id_X\colon X^3\longrightarrow X^3$
and $r_2=\id_X\times r\colon X^3\longrightarrow X^3$.
\end{itemize}

Let $X,Y$ be non-empty sets and let $r\colon X^2\longrightarrow
X^2$ and $r'\colon Y^2\longrightarrow Y^2$ be maps such that
$r(x,x')=(\sigma_x(x'), \gamma_{x'}(x))$ and
$r'(y,y')=(\sigma'_{y}(y'), \gamma'_{y'}(y))$. A homomorphism from
$(X,r)$ to $(Y,r')$ is a map $f\colon X\longrightarrow Y$ such
that
$$(f(\sigma_x(x')),f(\gamma_{x'}(x)))=r'(f(x),f(x')),$$
for all  $x,x'\in X$. Note that if $r$ is a (right
non-degenerate involutive) set-theoretic solution of the Yang-Baxter equation
then so is $(f(X),r'_{|f(X)^{2}})$. If $f$ is bijective then we say
that $(X,r)$ and $(Y,r')$ are isomorphic.

The  next result is an easy generalization of \cite[Theorem
4.1]{CJO2}.

\begin{proposition}\label{solution}
Let $X$ be a non-empty set. Let $r\colon X^2\longrightarrow X^2$ be
a map such that $r(x,y)=(\sigma_x(y), \gamma_y(x))$. Then $(X,r)$ is
an  right non-degenerate involutive set-theoretic solution of the
Yang-Baxter equation if and only if
\begin{itemize}
\item[(i)] $r^2=\id_{X^2}$,
\item[(ii)] $\sigma_x\in \Sym_X$, for all $x\in X$,
\item[(iii)] $\sigma_x\circ \sigma_{\sigma_x^{-1}(y)}=
\sigma_y\circ \sigma_{\sigma_y^{-1}(x)}$, for all $x,y\in X$.
\end{itemize}
\end{proposition}
\begin{proof} It is straightforward (see the proof of \cite[Theorem
9.3.10]{JObook}).
\end{proof}

Let $(X,r)$ be a non-degenerate involutive set-theoretic solution of
the Yang-Baxter equation.  The structure group of $(X,r)$ is the
group $G(X,r)$  generated by  the elements of the set $X$ and
subject to the  relations $xy=zt$ for all $x,y,z,t\in X$ such that
$r(x,y)=(z,t)$. Let $\mathbb{Z}^X$ denote the additive free abelian
group with basis $X$. Consider the natural action of $\Sym_X$ on
$\mathbb{Z}^X$ and the associated semidirect product
$\mathbb{Z}^X\rtimes \Sym_X$. In \cite[Propositions 2.4 and
2.5]{ESS} it is proved that $G(X,r)$ is isomorphic to a subgroup of
$\mathbb{Z}^X\rtimes \Sym_X$ of the form
$$\{ (a,\phi(a))\mid a\in \mathbb{Z}^X \},$$
for some function $\phi\colon \mathbb{Z}^X\longrightarrow \Sym_X$.
In fact, $\phi(x)=\sigma_x$ for all $x\in X$, where
$r(x,y)=(\sigma_x(y),\gamma_y(x))$, and the map $x\mapsto
(x,\sigma_x)$ from $X$ to $\mathbb{Z}^X\rtimes \Sym_X$ induces an
isomorphism from $G(X,r)$ to $\{ (a,\phi(a))\mid a\in \mathbb{Z}^X
\}$.  The subgroup of $\Sym_X$ generated by $\{\sigma_x\mid x\in
X\}$ is denoted by   $\mathcal{G}(X,r)$ (see \cite{GC}).

We now recall the definition  of a monoid of $I$-type introduced by
Gateva-Ivanova and Van den Bergh in \cite{GIVdB}. For this, let
$\FaM_n$ denote the multiplicative free abelian monoid of rank $n$
with basis $\{ u_1,\dots ,u_n\}$.

\begin{definition}
A monoid $S$ generated by a set $X=\{ x_1,\dots ,x_n\}$ is said to
be of left $I$-type if there exists a bijection (called a left
$I$-structure)
$$v\colon \FaM_n\longrightarrow S$$
such that, for all $a\in \FaM_n$,
$$v(1)=1 \quad\mbox{and}\quad \{ v(u_1a),\dots ,v(u_na)\}=\{ x_1v(a),\dots ,x_nv(a)\}.$$
\end{definition}

Similarly, one defines monoids of right $I$-type.

In \cite[Theorem 1.3]{GIVdB} it is proved that a monoid $S$ is of
left $I$-type if and only if there exists a left
non-degenerate involutive set-theoretic solution of the Yang-Baxter equation
$(X,r)$ such that $X$ is finite and $S$ is isomorphic to the monoid
generated by the elements of the set $X$ and subject to the
relations $xy=zt$, where $x,y,z,t\in X$ such that $r(x,y)=(z,t)$.
Furthermore, it is also proved that $S$ has a group of quotients
$G=SS^{-1}$ and $G$ is a Bieberbach group \cite[Theorem 1.6]{GIVdB}.

In \cite{JO} Jespers and Okni\'{n}ski proved that a monoid $S$ is of
left $I$-type if and only if $S$ is of right $I$-type. Therefore if
$(X,r)$ is a right non-degenerate involutive set-theoretic solution
of the Yang-Baxter equation and $X$ is finite, then $(X,r)$ also is
left non-degenerate, that is, if $r(x,y)=(\sigma_x(y),
\gamma_y(x))$, then $\gamma_x\in \Sym_X$, for all $x\in X$.
Furthermore $\sigma_x(y)=\gamma^{-1}_{\gamma_y(x)}(y)$ and
$\gamma_x(y)=\sigma^{-1}_{\sigma_y(x)}(y)$ for all $x,y\in X$. In
particular, the group of fractions of $S$ is  isomorphic to the
structure group $G(X,r)$ associated to the solution $(X,r)$. The
monoid $S$ is simply called a monoid of $I$-type and its group of
fractions $SS^{-1}$ is said to be a group of $I$-type.

 A finite group $G$
is said to be an involutive Yang-Baxter group (IYB group)  if it is
isomorphic to $\mathcal{G}(X,r)$ for some
non-degenerate involutive set-theoretic solution $(X,r)$ of the Yang-Baxter
equation with $X$ finite (see  \cite{CJR}).

In order to  describe and to study the non-degenerate involutive
set-theoretic finite solutions of the Yang-Baxter equation one needs
to investigate  which groups are $I$-type groups and determine their
structure. This naturally leads to the question which groups are IYB
groups.

In \cite[Theorem 2.15]{ESS}, it is shown that if $X$ is finite, then
$G(X,r)$ is a solvable group. Note that the map
$$G(X,r)\cong \{ (a,\phi(a))\mid a\in \mathbb{Z}^X \}\longrightarrow
\Sym_X$$ induced by $x\mapsto (x,\sigma_x)\mapsto \sigma_x$, for
$x\in X$ shows that $\mathcal{G}(X,r)$ is a homomorphic image of
$G(X,r)$. Thus, if $X$ is finite, then the IYB group
$\mathcal{G}(X,r)$ is solvable. It is not known whether  every
finite solvable group is IYB (see \cite{CJR}).

In order to study the structure of  a  group of $I$-type and to
classify set-theoretic solutions, Etingof, Shedler and Soloviev in
\cite{ESS} proposed an  interesting   operator as a tool. We recall
its definition. Let $(X,r)$ be a non-degenerate involutive
set-theoretic solution of the Yang-Baxter equation. Suppose that
$r(x,y)=(\sigma_x(y),\gamma_y(x))$. The retract relation $\sim $ on
the set $X$  with respect to $r$ is defined by $x\sim y$ if
$\sigma_{x}=\sigma_{y}$. There is a natural induced solution
$\Ret(X,r)=(X/\!\sim, \tilde{r})$, called the retraction of $(X,r)$,
that is
$$\tilde{r}([x],[y])=([\sigma_x(y)],[\gamma_y(x)]),$$
where $[x]$ denotes the $\sim$-class of $x\in X$. A  non-degenerate
involutive set-theoretic solution $(X,r)$ is called a
multipermutation solution of level $m$ if $m$ is the smallest
nonnegative integer such that the solution $\Ret^{m}(X,r)$ has
cardinality $1$; in this case we write $\mpl(X,r)=m$. Here we define
$\Ret^{k}(X,r)=\Ret(\Ret^{k-1}(X,r))$ for $k>1$. If such an $m$
exists then one also says that the solution is retractable. In this
case, if $X$ also is finite then it is easy to see that the group
$G(X,r)$ is a poly-(infinite cyclic) (Proposition 4.2 in \cite{JO}).
Note that  not all groups of $I$-type are poly-(infinite cyclic)
(see \cite{JO}).

\section{Relations between the Yang-Baxter equation and braces} \label{relationsYB}

We now discuss a strong connection between braces and non-degenerate
involutive  set-theoretic solutions of the Yang-Baxter equation. We
begin with the following lemma  which is implicit in \cite{rump3}
and in \cite[page 132]{rump2}.

\begin{lemma}\label{varphi}
Let $G$ be a left brace. The following properties hold.
\begin{itemize}
\item[(i)]
$a\lambda_a^{-1}(b)=b\lambda_b^{-1}(a)$.
\item[(ii)]
$\lambda_a\lambda_{\lambda_a^{-1}(b)}=\lambda_b\lambda_{\lambda_b^{-1}(a)}$.
\item[(iii)] The map $r\colon G\times G\longrightarrow G\times G$
defined by $r(x,y)=(\lambda_x(y),\lambda_{\lambda_x(y)}^{-1}(x))$ is
a non-degenerate involutive  set-theoretic solution of the
Yang-Baxter equation.
\end{itemize}
\end{lemma}

\begin{proof}
$(i)$ Let $a,b\in G$. Then
\begin{eqnarray*}
a\lambda_a^{-1}(b) &=&a(a^{-1}(b+a))=b+a =b\lambda_b^{-1}(a).
\end{eqnarray*}
$(ii)$ is a consequence of $(i)$ and Lemma~\ref{lambda}.

\noindent $(iii)$ It is easy to check that $r^2=\id_{G^2}$. By
$(ii)$ and Proposition~\ref{solution} $(G,r)$ is a
right non-degenerate involutive set-theoretic solution of the Yang-Baxter
equation. Let $g_y(x)=\lambda_{\lambda_x(y)}^{-1}(x)$. To show that
$(G,r)$ is left non-degenerate we prove that $g_y$ is bijective. For
this we note that
\begin{eqnarray*}
g_y(x)&=&\lambda_{\lambda_x(y)}^{-1}(x)=\lambda_{xy-x}^{-1}(x)=(xy-x)^{-1}(x+xy-x)\\
&=&(xy-x)^{-1}xy=((xy)^{-1}(xy-x))^{-1}=(1-y^{-1}+(xy)^{-1})^{-1}\\
&=&(-y^{-1}+(xy)^{-1})^{-1}.
\end{eqnarray*}
Therefore the assertion follows.
\end{proof}

Let $G$ be a left brace. The set-theoretic solution of the
Yang-Baxter equation $(G,r)$ defined in Lemma~\ref{varphi} is called
the solution of the Yang-Baxter equation associated to the left
brace $G$.

We now give an equivalent condition for a group to be a
multiplicative group of a left brace.  This is an easy
generalization of a part of \cite[Theorem 2.1]{CJR} (see the
definition of IYB morphism in \cite[page 2546]{CJR}).

\begin{proposition} \label{lambda-main}
A group $G$ is the  multiplicative group of a left brace if and only
if there exists a group homomorphism $\mu: G\longrightarrow
\Sym_{G}$ such that $x\mu (x)^{-1}(y)=y\mu (y)^{-1}(x)$ for all
$x,y\in G$.
\end{proposition}
\begin{proof}
Suppose that such a homomorphism $\mu$ is given. Define an operation
$+$ on $G$ by $x+y=x\mu (x)^{-1}(y)$ for $x,y\in G$. Clearly
$x+y=y+x$. Further, for $x,y,z\in G$,
\begin{eqnarray*}
(x+y)+z &=& y \mu (y)^{-1} (x) +z =  y \mu (y)^{-1} (x) \, \mu ( y \mu (y)^{-1} (x))^{-1}(z)\\
 &=&  y \mu (y)^{-1} (x) \; \mu ( \mu (y)^{-1} (x))^{-1} \; (\mu (y)^{-1}(z))\\
 &=&  y \mu (y)^{-1} (z) \; \mu (\mu (y)^{-1} (z))^{-1} \; (\mu (y)^{-1} (x))\\
 &=&  y \mu(y)^{-1}(z)\mu(y \mu (y)^{-1}(z))^{-1}(x)\\
 &=& (y+z)+x = x+ (y+z).
\end{eqnarray*}
It is easy to check that $1$ is the neutral element for this
addition and $\mu (x)(x^{-1})$ is the opposite of $x$.
Hence $(G,+)$ is an abelian group.

For $x,y,z\in G$ we also get that $x^{-1}\mu (x^{-1})^{-1}(y+z)=x^{-1}+y+z$ and hence
\begin{eqnarray*}
\mu (x) (y+z) &=& x\; (x^{-1}+y+z) \\
 &=& x \; ( x^{-1} \mu (x) (y) +z ) \\
 &=& xx^{-1} \mu (x) (y) \; \mu ( x^{-1}\mu (x) (y))^{-1} (z)\\
 &=& \mu (x) (y) \; \mu (\mu (x) (y) )^{-1} \; (\mu (x) (z))\\
 &=& \mu (x)(y) +\mu (x) (z)
\end{eqnarray*}
and therefore
\begin{eqnarray*}
x(y+z)+x&=&x\mu(x)^{-1}(\mu(x)(y+z))+x=x+\mu(x)(y+z)+x\\
&=&x+\mu(x)(y)+x+\mu(x)(z)=xy+xz.
\end{eqnarray*}
Hence $(G,+,\cdot)$ is a left brace. Note that in this brace
$\lambda_x=\mu(x)$ for every $x\in G$.

The converse follows by Lemmas~\ref{lambda} and~\ref{varphi}.
\end{proof}

\begin{theorem}\label{bracesol}
If $G$ is a left brace then there exists a
non-degenerate involutive set-theoretic solution $(X,r')$ of the Yang-Baxter
equation such that the solution $\Ret (X,r')$ is isomorphic to the
solution $(G,r)$ associated to the left brace $G$ and, moreover,
${\cal{G}} (X,r')$ is isomorphic to the multiplicative group of the
left brace $G$. Furthermore, if $G$ is finite then $X$ can be taken
a finite set.
\end{theorem}
\begin{proof}
Let $\eta\colon G\longrightarrow \Sym_G$ be the map defined by
$\eta(x)(y)=xy$. Note that $\eta$ is an injective homomorphism of
groups. Let $X=G\times\{ 1,2\}$. Consider the map $r'\colon X\times
X\longrightarrow X\times X$ defined by
$r'(x,y)=(f_x(y),f_{f_x(y)}^{-1}(x))$, where
$$f_{(a,1)}=\id_X,\quad f_{(a,2)}(b,1)=(ab,1),\quad\mbox{and}\quad f_{(a,2)}(b,2)=
(\lambda_a(b),2).$$
Note that $f_x$ is bijective for all $x\in X$. Now it is easy to
check that $f_xf_{f_x^{-1}(y)}=f_yf_{f_y^{-1}(x)}$ and that for each
$y$ the map $x\mapsto f_{f_x(y)}^{-1}(x)$ is bijective. Therefore,
by Proposition~\ref{solution}, $(X,r')$ is a
non-degenerate involutive set-theoretic solution of the Yang-Baxter equation.

Let $\psi\colon G\longrightarrow \Sym_X$ be the map defined by
$\psi(a)=f_{(a,2)}$. Since $\psi(a)(1,1)=(a,1)$, it is clear that
$\psi$ is injective. Note that
$$\psi(a)\psi(b)(c,1)=(abc,1)=\psi(ab)(c,1)$$
and
$$\psi(a)\psi(b)(c,2)=(\lambda_a\lambda_b(c),2)
=(\lambda_{ab}(c),2)=\psi(ab)(c,2).$$ Therefore $\psi$ is an
injective homomorphism of  groups. Furthermore,
$\mathcal{G}(X,r') = \langle f_{x} \mid x\in X \rangle = \langle f_{(a,2)}
\mid a\in G \rangle = \psi (G)$.

Let $\sim$ be the retract relation on $X$ associated to the solution
$(X,r')$. Let $(X/\!\sim, s)$ be the retraction of $(X,r')$. Let
$f\colon X=G\times \{ 1,2\}\longrightarrow G$ be the natural
projection. Note that $x\sim y$ if and only if $f_x=f_y$, and this
is also equivalent to $f(x)=f(y)$. Thus $\widetilde{f}\colon
X/\!\sim\longrightarrow G$, defined by $\widetilde{f}([x])=f(x)$,
for all $x\in X$, is a bijection (here $[x]$ denotes the
$\sim$-class of $x$). Furthermore
$$s([x],[y])=([f_x(y)],[f_{f_x(y)}^{-1}(x)]),$$
for all $x,y\in X$, and
\begin{eqnarray*}
r(\widetilde{f}([(a,2)]),\widetilde{f}([(b,2)]))&=&r(a,b)=(\lambda_a(b),\lambda_{\lambda_a(b)}^{-1}(a))\\
&=&(\widetilde{f}([(\lambda_a(b),2)]),\widetilde{f}([(\lambda_{\lambda_a(b)}^{-1}(a),2)]))\\
&=&(\widetilde{f}([f_{(a,2)}(b,2)]),\widetilde{f}([f_{f_{(a,2)}(b,2)}^{-1}(a,2)])).
\end{eqnarray*}
Since $[(a,1)]=[(1,2)]$ for every $a\in G$ the remaining
compatibility relations follow easily. Hence,
 the retraction of $(X,r')$  is
isomorphic to $(G,r)$.
\end{proof}

The following result is implicit in \cite{rump3}.

\begin{theorem}\label{YBsol}
Let $(X,s)$ be a non-degenerate involutive set-theoretic solution of
the Yang-Baxter  equation. Then $G(X,s)$ is isomorphic to the
multiplicative group of a left brace $H$ such that if $(H,r)$ is the
solution associated to it, then there exists a subset $Y$ of $H$
such that $(Y,r')$, where $r'$ is the restriction of $r$ to $Y^2$,
is a non-degenerate involutive set-theoretic solution of the
Yang-Baxter equation  that is isomorphic to $(X,s)$.
\end{theorem}

\begin{proof}
We know that $G(X,s)$ is isomorphic to a subgroup $H$ of
$\mathbb{Z}^X\rtimes \Sym_X$ of the form
$$\{ (a,\phi(a))\mid a\in \mathbb{Z}^X \},$$
for some function $\phi\colon \mathbb{Z}^X\longrightarrow \Sym_X$
such that $\phi(x)=\sigma_x$ for all $x\in X$, where
$s(x,y)=(\sigma_x(y),\gamma_y(x))$, and the map $x\mapsto
(x,\sigma_x)$ from $X$ to $\mathbb{Z}^X\rtimes \Sym_X$ induces an
isomorphism from $G(X,s)$ to $\{ (a,\phi(a))\mid a\in \mathbb{Z}^X
\}$. We define a sum in $H$ by
$$(a,\phi(a))+(b,\phi(b))=(a+b,\phi(a+b))$$
for all $a,b\in \mathbb{Z}^X$. Note that $(H,+)$ is an abelian group
isomorphic to $\mathbb{Z}^X$. Furthermore for $a,b,c\in
\mathbb{Z}^X$ we have
\begin{eqnarray*}
\lefteqn{
(a,\phi(a))((b,\phi(b))+(c,\phi(c)))+(a,\phi(a))}
\\
&=&(a,\phi(a))((b+c,\phi(b+c)))+(a,\phi(a))\\
&=&(a+\phi(a)(b)+\phi(a)(c),\phi(a+\phi(a)(b)+\phi(a)(c)))+(a,\phi(a))\\
&=&(a+\phi(a)(b)+a+\phi(a)(c),\phi(a+\phi(a)(b)+a+\phi(a)(c)))\\
&=&(a+\phi(a)(b),\phi(a+\phi(a)(b)))+(a+\phi(a)(c),\phi(a+\phi(a)(c)))\\
&=&(a,\phi(a))(b,\phi(b))+(a,\phi(a))(c,\phi(b)).
\end{eqnarray*}
Therefore $H$ is a left brace.

Let $Y=\{ (x,\phi(x))\mid x\in X\}=\{ (x,\sigma_x)\mid x\in X\}$.
For $x,y\in X$ we have
\begin{eqnarray*} \lefteqn{r((x,\sigma_x),(y,\sigma_y))}\\
&=&(\lambda_{(x,\sigma_x)}(y,\sigma_y),\lambda_{\lambda_{(x,\sigma_x)}(y,\sigma_y)}^{-1}(x,\sigma_x))\\
&=&((x,\sigma_x)(y,\sigma_y)-(x,\sigma_x),\lambda_{\lambda_{(x,\sigma_x)}(y,\sigma_y)}^{-1}(x,\sigma_x))\\
&=&((x+\sigma_x(y),\phi(x+\sigma_x(y)))-(x,\sigma_x),\lambda_{\lambda_{(x,\sigma_x)}(y,\sigma_y)}^{-1}(x,\sigma_x))\\
&=&((\sigma_x(y),\sigma_{(\sigma_x(y))}),\lambda_{(\sigma_x(y),\sigma_{(\sigma_x(y))})}^{-1}(x,\sigma_x))\\
&=&((\sigma_x(y),\sigma_{(\sigma_x(y))}),
(\sigma_x(y),\sigma_{(\sigma_x(y))})^{-1}(x,\sigma_x)-(\sigma_x(y),\sigma_{(\sigma_x(y))})^{-1})\\
&=&((\sigma_x(y),\sigma_{(\sigma_x(y))}),
(-\sigma_{(\sigma_x(y))}^{-1}(\sigma_x(y)),\sigma_{(\sigma_x(y))}^{-1})(x,\sigma_x)
-(\sigma_x(y),\sigma_{(\sigma_x(y))})^{-1})\\
&=&((\sigma_x(y),\sigma_{(\sigma_x(y))}),
(\sigma_{(\sigma_x(y))}^{-1}(x),\sigma_{\sigma_{(\sigma_x(y))}^{-1}(x)})).
\end{eqnarray*}
Since
$$s(x,y)=(\sigma_x(y),\sigma_{(\sigma_x(y))}^{-1}(x)),$$
we have that the map $x\mapsto (x,\sigma_x)$ is an isomorphism from
$(X,s)$ to $(Y,r')$, where $r'$ is the restriction of $r$ to
$Y^{2}$.
\end{proof}

\begin{remark}\label{braceIYB}
{\rm Let $(X,s)$ be a non-degenerate involutive set-theoretic
solution of the Yang-Baxter  equation. We have seen in the proof of
Theorem~\ref{YBsol} that  $G(X,s)$ is isomorphic to the
multiplicative group of the left brace $H=\{ (a,\phi(a))\mid a\in
\mathbb{Z}^X \}$. We know that $\phi(x)=\sigma_x$ for all $x\in X$,
where
$$s(x,y)=(\sigma_x(y),\gamma_y(x)).$$
Thus the map $\pi\colon H\longrightarrow \Sym_X$ defined by
$\pi(a,\phi(a))=\phi(a)$, for all $a\in \mathbb{Z}^X$, is a
homomorphism from the multiplicative group of $H$ to $\Sym_X$ such
that $\pi(H)=\mathcal{G}(X,s)$ and $\Ker(\pi)=\{ (a,\phi(a))\in
H\mid \phi(a)=\id_X\}$. Note that for all $(a,\id_X)\in \Ker(\pi)$
and $(b,\phi(b))\in G$ we have,  by Lemma~\ref{lambda} and
Lemma~\ref{varphi},
\begin{eqnarray*}
\lambda_{(b,\phi(b))}(a,\id_X)&=&(b,\phi(b))(a,\id_X)\lambda_{(-a,\id_X)}((b,\phi(b))^{-1})\\
&=&(b+\phi(b)(a),\phi(b))((-a+\phi(b)^{-1}(-b),\phi(b)^{-1})+(a,\id_X))\\
&=&(b+\phi(b)(a),\phi(b))(\phi(b)^{-1}(-b),\phi(b)^{-1})\\
&=&(\phi(b)(a),\id_X)\in \Ker(\pi).
\end{eqnarray*}
Hence $\Ker(\pi)$ is an ideal of the left brace $H$. Therefore
$\mathcal{G}(X,s)$ is isomorphic to the  multiplicative group of the
left brace $H/\Ker(\pi)$. }
\end{remark}

An easy consequence of  Theorem~\ref{bracesol} and
Remark~\ref{braceIYB} is the following result (see \cite[Theorem
2.1]{CJR}).
\begin{corollary} \label{finite-group-IYB}
A finite group $G$ is an IYB group if and only if it is the
multiplicative group of a finite left brace.
\end{corollary}

\section{Groups of $IG$-type and braces} \label{groups}

In this section we give a natural generalization of groups of
$I$-type (considered as subgroups of semidirect products
$\mathbb{Z}^{X}\rtimes \Sym_{X}$) and we show that these are in a
one-to-one correspondence with left braces.  A connection of this
type in the context of regular permutation subgroups of the affine
group $AGL(V)$ on a vector space $V$ was given by Catino and Rizzo
\cite{CatRiz}.

\begin{theorem}
Let $(A,+)$ be an  abelian group. Let
$$\mathcal{B}(A) =\{ (A,+,\cdot ) \mid (A,+,\cdot) \mbox{ is a left brace} \}$$ and
$$\mathcal{S}(A) =\{ G \mid G \mbox{ is a subgroup of } A\rtimes \Aut (A) \mbox{ of the form }
G=\{ (a, \phi (a) ) \mid a\in A\} \} .$$ The map $f:\mathcal{B}(A)
\rightarrow \mathcal{S}(A)$  defined by
 $$f (A,+,\cdot ) =\{ (a,\lambda_{a}) \mid a\in A\}$$
 is  bijective.
\end{theorem}
\begin{proof}
Let $a,b\in A$. Then
$$(a,\lambda_a)(b,\lambda_b)=(a+\lambda_{a}(b),\lambda_a\lambda_b)=(a+ab-a, \lambda_{ab})= (ab,\lambda_{ab})$$
and \begin{eqnarray*} (a,\lambda_a)^{-1}&=& (\lambda_{a}^{-1}(-a) ,
\lambda_{a}^{-1}) = (-\lambda_{a^{-1}}(a) , \lambda_{a^{-1}})\\
&=& ( -a^{-1}a+a^{-1}, \lambda_{{a}^{-1}}) =(a^{-1} ,
\lambda_{a^{-1}}).\end{eqnarray*} Hence $\{ (a,\lambda_{a}) \mid
a\in A\} \in \mathcal{S}(A)$ and thus $f$ is well defined.

Consider $G=\{ (a, \phi (a)) \mid a\in A\} \in \mathcal{S}(A)$. We
define an operation $\cdot$ on $A$ by
 $$a\cdot b = a+\phi (a) (b),$$
for $a,b\in A$. We prove that $(A,+,\cdot )$ is a left brace.  For
this we first show that $(A,\cdot )$ is a group. Let $a,b,c\in A$.
Since $(a, \phi (a)) \, (b,\phi (b)) = (a+\phi (a) (b), \phi (a) \,
\phi (b))\in G$ we have that $\phi( a+\phi (a) (b)) =\phi (a)\, \phi
(b)$. Thus $(a\cdot b ) \cdot c =( a+\phi (a)(b)) \cdot c = a+\phi
(a) (b) +\phi (a+\phi (a)(b))(c) =a+\phi (a)(b) +\phi (a)\, \phi (b)
(c)$ and $a\cdot (b  \cdot c)  = a\cdot ( b+\phi (b)(c) ) =a+ \phi
(a) (b+\phi (b)(c)) = a+ \phi (a)(b) + \phi (a) \, \phi (b)(c)$. So
$(a\cdot b ) \cdot c =a\cdot (b \cdot c )$. It is easily verified
that $0$ is the identity element of the semigroup $(A,\cdot )$.
Notice $\phi (a)^{-1}(-a)=a^{-1}$ in $(A,\cdot)$. Hence, $(A,\cdot
)$ is a group. We also have
$$a\cdot (b+c) + a = a +\phi (a) (b+c) + a = a + \phi (a)(b) +a +\phi (a) (c) =
a\cdot b + a\cdot c .$$ Hence  $(A,+,\cdot )$ is a left brace.

 So, we have constructed a map $\mathcal{S}(A) \rightarrow \mathcal{B}(A)$ and it
 is easy to verify that this is the inverse of the map $f$.
\end{proof}

Note that if $G=\{ (a,\phi(a)) \mid a\in A\} \in \mathcal{S}(A)$
then from the proof it follows that the bijective map $G\rightarrow
A: (a,\phi (a)) \mapsto a$ is a multiplicative homomorphism of the
group $G$ to the  multiplicative group of the left brace
$f^{-1}(G)$. We can now define a sum on $G$ using this bijection,
that is $(a,\phi (a)) + (b,\phi (b)) =(a+b, \phi (a+b))$ for $a,b\in
A$, making this bijection an isomorphism of left braces.

In case  the group  $\sigma(A)$ is finite then the   group $G$ has
been investigated in \cite{goffa} and it was  called a group of
$IG$-type. Thus every group of $IG$-type is isomorphic to the
multiplicative group of a left brace. Obviously, groups of $I$-type
are also of $IG$-type.

 We can now  characterize
groups of $I$-type as follows.

\begin{proposition} \label{Itype}  A group $G$ is of $I$-type if and only if
it is isomorphic to the multiplicative group of a left brace $B$
such that the additive group of $B$ is a free abelian group with a
finite basis $X$ such that $\lambda_{x}(y)\in X$ for all $x,y\in X$.
\end{proposition}

\begin{proof}
Suppose $G$ is a group of $I$-type. We may assume that $G$ is a
subgroup of $\mathbb{Z}^Y\rtimes\Sym_Y$, for some finite set $Y$, of the form
$$G= \{ (u,\phi(u))\mid u\in \mathbb{Z}^Y\},$$
and  for some mapping $\phi\colon \mathbb{Z}^Y\longrightarrow
\Sym_Y$. By the comment given before the proposition, we know that
$G$ is a left brace for the addition
$$(a,\phi(a))+(b,\phi(b))=(a+b,\phi(a+b)),$$
with  $a,b\in \mathbb{Z}^Y$.
Its additive
group is a free abelian group with basis $X=\{ (x,\phi(x))\mid x\in
Y\}$ and
\begin{eqnarray*}
\lambda_{(x,\phi(x))}(y,\phi(y))&=&
(x,\phi(x))(y,\phi(y))-(x,\phi(x))\\
&=&(x+\phi(x)(y),\phi(x+\phi(x)(y)))-(x,\phi(x))\\
&=&(\phi(x)(y),\phi(\phi(x)(y)))\in X.
\end{eqnarray*}

Conversely, suppose that $B$ is a left brace such  that its additive
group  is a free abelian group with a finite basis $X$ such that
$\lambda_{x}(y)\in X$ for all $x,y\in X$. Hence
$\lambda_{x}^{-1}(y)\in X$, for all $x,y\in X$. By
Lemma~\ref{lambda}, $\lambda_{x}^{-1}=\lambda_{x^{-1}}$. In
particular, $\lambda_{x^{-1}}(x)=1-x^{-1}=-x^{-1}\in X$ for all
$x\in X$. Since $X$ is finite, the map $x\mapsto -x^{-1}$ is a
permutation of $X$. It follows that $(-x)^{-1}\in X$ (the inverse
image of $x$) for all $x\in X$. Therefore,
$\lambda_{-x}(y)=\lambda_{(-x)^{-1}}^{-1}(y)\in X$ for all $x,y\in
X$. Let $a\in B$. We know from Lemma~\ref{lambda} that $\lambda_a$
is an automorphism of the additive group of the brace $B$.

Suppose that $a=\sum_{x\in X} z_xx$ for some integers $z_x$. We
shall prove that $\lambda_a(x)\in X$, for all $x\in X$, by induction
on $m_a=\sum_{x\in X}|z_x|$. If $m_a=1$, then the result is clear.
Suppose that $m_a>1$ and that $\lambda_b(x)\in X$, for all $x\in X$
and all $b\in B$ such that $m_b<m_a$. Clearly there exists $x\in
X$ such that either $m_{a+x}<m_a$ or $m_{a-x}<m_a$. Note that, for
all $c,d\in B$,
\begin{eqnarray} \label{sum-lambda}
c+d=c+d-1=c+cc^{-1}d-cc^{-1}=c(c^{-1}d-c^{-1})=c\lambda_{c^{-1}}(d).
\end{eqnarray}

Suppose that $m_{a+x}<m_a$. Because of (\ref{sum-lambda}) we get that
$$\lambda_{a}=\lambda_{a+x-x}=\lambda_{(-x)\lambda_{-x}^{-1}(a+x)}
=\lambda_{-x}\lambda_{\lambda_{-x}^{-1}(a+x)}.$$ Since
$\lambda_{-x}(y)\in X$ and $\lambda_{-x}$ is an automorphism of the
additive group of $B$, we have that
$m_{\lambda_{-x}^{-1}(a+x)}=m_{a+x}<m_a$. Hence by induction
hypothesis, $\lambda_{\lambda_{-x}^{-1}(a+x)}(y)\in X$, and
therefore $\lambda_a(y)\in X$, in this case.

Suppose that $m_{a-x}<m_a$. Again by (\ref{sum-lambda}) we have
$$\lambda_{a}=\lambda_{a-x+x}=\lambda_{x\lambda_{x}^{-1}(a-x)}
=\lambda_{x}\lambda_{\lambda_{x}^{-1}(a-x)}.$$ Since
$\lambda_{x}(y)\in X$ and $\lambda_{x}$ is an automorphism of the
additive group of $B$, we have that
$m_{\lambda_{x}^{-1}(a-x)}=m_{a-x}<m_a$. Hence by induction
hypothesis, $\lambda_{\lambda_{x}^{-1}(a-x)}(y)\in X$, and therefore
$\lambda_a(y)\in X$, in this case.

Thus every automorphism $\lambda_a$ of the additive group of $B$ is
induced by a permutation of the finite set $X$. Note that for all
$a,b\in B$
$$(ab,\lambda_{ab})=(a+\lambda_a(b),\lambda_a\lambda_b)=(a,\lambda_a)(b,\lambda_b).$$
Therefore the map $f:B\longrightarrow \{(a,\lambda_a)\mid a\in B \}$
defined by $f(a)=(a,\lambda_a)$ is an isomorphism of the
multiplicative group of $B$ and the subgroup $\{(a,\lambda_a)\mid
a\in B \}$ of the semidirect product $B\rtimes \Aut(B,+)$. Hence the
 multiplicative group of $B$ is of $I$-type.
\end{proof}

Note also that if $G$ is a left brace, then the identity map $\id
\colon G\longrightarrow G$ is a bijective $1$-cocycle of the group
$G$ with coefficients in the left $G$-module $(G,+)$ with respect
the action given by $\lambda\colon G\longrightarrow \Aut (G,+).$
This is because for all $a,b\in G$ we have $ab=a+\lambda_a(b).$
Furthermore, if $H$ is a group, $A$ is a left $H$-module and
$\pi\colon H\longrightarrow A$ is a bijective $1$-cocycle, then we
can define a sum $+$ in $H$ by
$$a+b=\pi^{-1}(\pi(a)+\pi(b)),$$
for all $a,b\in H$; and the group $H$ with this sum is a left brace.
This is because for all $a,b,c\in H$ we have that
\begin{eqnarray*}
\pi(a(b+c)+a)&=&\pi(a(b+c))+\pi(a)\\
&=&\pi(a)+ a\pi(b+c)+\pi(a)\\
&=&\pi(a)+ a(\pi(b)+\pi(c))+\pi(a)\\
&=&\pi(a)+a\pi(b)+a\pi(c)+\pi(a)\\
&=&\pi(ab)+\pi(ac).
\end{eqnarray*}
This gives us a bijective correspondence between left braces and
groups with a bijective $1$-cocycle with respect to a left action (see
\cite{CJR,ESS}). The latter class of groups gives rise to  groups of central
type which play an important role in the theory of finite dimensional central
simple algebras (\cite{DavidGinosar,EtingofGelaki}). This class
of groups  with a bijective $1$-cocycle also appeared in the context
of left-invariant  affine structures on Lie groups \cite{Burde}.

\section{Socle of a brace and multipermutation solutions} \label{soclemulti}

Let $G$ be a left brace. The socle of $G$ is
$$\Soc (G)=\{ a\in G\mid \lambda_a=\id\}=\{ a\in G\mid ab=a+b\;\mbox{ for all }\; b\in G\}.$$
By Lemma~\ref{lambda}, we have that $\Soc(G)$ is a normal subgroup
of the multiplicative group of $G$. Let $a\in G$ and $b\in \Soc(G)$.
Then by Lemma~\ref{varphi},
$$\lambda_a\lambda_{\lambda_a^{-1}(b)}=\lambda_b\lambda_{\lambda_b^{-1}(a)}=\lambda_a.$$
Hence $\lambda_{\lambda_a^{-1}(b)}=\id$ and thus
$\lambda_a^{-1}(b)\in\Soc(G)$. Therefore $\Soc(G)$ is an ideal of
the left brace $G$.

\begin{lemma} \label{brace-soc}  (Rump \cite[Proposition
7]{rump3}) Let $G$ be a left brace and  let $(G,r)$ be the solution
of the Yang-Baxter equation associated to $G$ (as in
Lemma~\ref{varphi}). If $(G/\Soc (G), r')$ is the solution of the
Yang-Baxter equation associated to the left brace $G/\Soc (G)$ then
$(G/\Soc (G), r') = \Ret (G,r)$.
\end{lemma}
\begin{proof}
Recall that the retract relation $\sim$ associated to the solution
$(G,r)$ is defined by $a\sim b $ if and only if $\lambda_{a}
=\lambda_{b}$. Then the $\sim$-equivalence class of $a\in G$ is
 \begin{eqnarray*}
 \{ b\in G \mid \lambda_{a}=\lambda_{b}\} &=& \{ b\in G \mid \lambda_{a} \lambda_{b}^{-1} =\id \}\\
  &=& \{ b\in G \mid \lambda_{ab^{-1}} =\id \}\\
  &=& \{ b \in G \mid ab^{-1} \in \Soc (G)\}
 \end{eqnarray*}
Hence, as sets, $G/\sim \, = \, G/\Soc (G)$. Let $(\widetilde{G},
\widetilde{r}) =\Ret (G,r)$. Then $\widetilde{G} =G/\Soc (G)$.
Denote by $\bar{a}$ the element $a\Soc (G) \in G/\Soc (G)$. Recall
from Lemma~\ref{varphi} that $r(a,b)= (\lambda_{a}(b),
\lambda_{\lambda_{a}(b)}^{-1} (a) )$. Hence
\begin{eqnarray*}
\widetilde{r}(\bar{a},\bar{b}) &=& (\overline{\lambda_{a}(b)}, \overline{\lambda_{\lambda_{a}(b)}^{-1} (a)})\\
 &=& (\overline{ab-a}, \overline{(ab-a)^{-1}a-(ab-a)^{-1}})\\
 &=& (\bar{a}\bar{b} -\bar{a}, (\bar{a}\bar{b}-\bar{a})^{-1}\bar{a}-(\bar{a}\bar{b}-\bar{a})^{-1}\\
 &=& (\lambda_{\bar{a}}(\bar{b}), \lambda_{\lambda_{\bar{a}}(\bar{b})}^{-1} (\bar{a}))
\end{eqnarray*}
Hence, $\Ret (G,r) = (\widetilde{G},\widetilde{r}) =(G/\Soc (G),
r')$ is the solution of the Yang-Baxter equation associated to the
left brace $G/\Soc (G)$.
\end{proof}

The following result is implicit in \cite[page 154]{rump3}.

\begin{proposition}\label{radical}
If $G$ is a finite non-trivial two-sided brace, then $\Soc(G)\neq \{
1\}$. Furthermore, the solution $(G,r)$ of the Yang-Baxter equation
associated to $G$ is a multipermutation solution.
\end{proposition}
\begin{proof}
From Proposition~\ref{rad} we know that $(G,+,*)$ is a a radical
ring, where $*$ is the operation on $G$ defined by $a*b=ab-a-b$, for
$a,b\in G$. Since $G$ is finite and non-trivial, it has  a minimal
right ideal,  say  $I$.  Because $G$ also is a radical ring, $I*G=\{
1\}$ (recall that $1$ is the zero element of the ring $(G,+,*)$).
Thus if $a\in I$ and $b\in G$, then
$$\lambda_a(b)=ab-a=a*b+b=b.$$
Therefore $I\subseteq\Soc(G)$, and the first part of the result follows.

By Lemma~\ref{brace-soc}, $\Ret (G,r) =(G/\Soc (G), r')$ is the solution of the Yang-Baxter
equation associated to the two-sided brace $G/\Soc (G)$. Therefore the second part of the
result follows by induction on the cardinality of $G$.
\end{proof}

In order to prove the main result of this section we need the
following two lemmas.

\begin{lemma} \label{reduction}
Let $(X,r),(Y,s)$ be two non-degenerate involutive set-theoretic
solutions of the Yang-Baxter equation. Let $f:X\longrightarrow Y$ be
an onto map that is a homomorphism from $(X,r)$ onto $(Y,s)$. Then
there exists an onto map $\widetilde{f}:\widetilde{X}\longrightarrow
\widetilde{Y}$ which is a homomorphism from
$(\widetilde{X},\widetilde{r})=\Ret(X,r)$ onto
$(\widetilde{Y},\widetilde{s})=\Ret(Y,s)$ such that the diagram
\[ \xymatrix{
X \ar@{->}^(0.5){\pi_{X}}[r] \ar_{f}[d] &
\widetilde{X} \ar^{\tilde{f}}[d]\\
Y \ar@{->}^(0.5){\pi_{Y}}[r] & \widetilde{Y} }\] is commutative,
where $\pi_{X},\pi_{Y}$ are the natural maps from $(X,r)$ to
$\Ret(X,r)$ and from $(Y,s)$ to $\Ret(Y,s)$, respectively. Moreover,
if $(X,r)$ is a multipermutation solution then $(Y,s)$ is a
multipermutation solution.
\end{lemma}
\begin{proof}
Let $r(x,y)=(\sigma_x(y),\gamma_y(x))$ and
$s(z,t)=(\sigma'_z(t),\gamma'_t(z))$ for $x,y\in X$ and $z,t\in Y$.
We  define $\widetilde{f}\colon\widetilde{X}\longrightarrow
\widetilde{Y}$ by $\widetilde{f}(\pi_X(x))=\pi_Y(f(x))$ for $x\in
X$. Let $x,y\in X$ be such that $\pi_X(x)=\pi_X(y)$. Then
$\sigma_x=\sigma_y$ and since $f$ is a homomorphism, we have
$$\sigma'_{f(x)}(f(u))=\sigma'_{f(y)}(f(u)),$$ for all $u\in X$.
Since $f$ is onto, we get that $\sigma'_{f(x)}=\sigma'_{f(y)}$ and
thus $\pi_Y(f(x))=\pi_Y(f(y))$. Therefore $\widetilde{f}$ is
well-defined and clearly it is onto and the diagram in the statement
of the lemma commutes.
 It is easy to check  that $\widetilde{f}$ is a homomorphism from $\Ret(X,r)$ to $\Ret(Y,s)$.
The last assertion follows by induction on the multipermutation
level of $(X,r)$.
\end{proof}

\begin{lemma} \label{reduction2}
Let $(X,r)$ be a non-degenerate involutive set-theoretic solution
of the Yang-Baxter equation. Assume that $Y$ is an
$\mathcal{G}(X,r)$-invariant subset of $X$. If $(X,r)$ is a
multipermutation solution then $(Y,r_{|Y^{2}})$ also is a
multipermutation solution.
\end{lemma}
\begin{proof}
 Since $Y$ is a $\mathcal{G}(X,r)$-invariant subset of $X$, it is
clear that $(Y,r_{|Y^2})$ is a non-degenerate involutive
set-theoretic solution of the Yang-Baxter equation. We shall see
that it is a multipermutation solution by induction on the
multipermutation level $\mpl(X,r)$ of $(X,r)$. If $\mpl(X,r)\leq 1$
then the assertion is clear. Assume that $m=\mpl(X,r)>1$. Let
$\pi_{X}: X\longrightarrow \widetilde{X}$ be the natural map from
$(X,r)$ to $\Ret(X,r)=(\widetilde{X}, \widetilde{r})$.  Let
$Z=\pi_{X}(Y)$ and $s=\widetilde{r}_{|Z^{2}}$.  Since $Y$ is a
$\mathcal{G}(X,r)$-invariant subset of $X$,  $Z$ is a
$\mathcal{G}(\widetilde{X},\widetilde{r})$-invariant subset of $X$.
Then $(Z,s)$  also is a non-degenerate involutive set-theoretic
solution of the Yang-Baxter equation. Let $\pi_{Y}$ be the natural
homomorphism from $(Y,r_{|Y^{2}})$ to
$\Ret(Y,r_{|Y^{2}})=(\widetilde{Y},t)$. We define a map
$f:Z\longrightarrow \widetilde{Y}$ by $f(\pi_{X}(y))=\pi_{Y}(y)$ for
$y\in Y$. It is easy to check that $f$ is well-defined and it is an
onto homomorphism from $(Z,s)$ to $(\widetilde{Y},t)$. Since
$(\widetilde{X},\widetilde{r})$ is a multipermutation solution of
level $m-1$, by induction  hypothesis $(Z,s)$ is a multipermutation
solution and then by Lemma~\ref{reduction} it follows that
$(\widetilde{Y},t)$ is a multipermutation solution. Therefore
$(Y,r_{|Y^{2}})$ is a multipermutation solution.
\end{proof}

Now we are ready to prove the following result that partially
generalizes \cite[Corollary 2.9]{CJO} and extends Theorem 2.5 in
\cite{CJO} to the class of all finite non-degenerate involutive
set-theoretic solutions. The method of proof is completely
different.

\begin{theorem}\label{abelian}
Let $(X,s)$ be a finite non-degenerate involutive set-theoretic
solution of the Yang-Baxter equation. If its associated IYB group
$\mathcal{G}(X,s)$ is abelian, then $(X,s)$ is a multipermutation
solution.
\end{theorem}

\begin{proof}
Let $G(X,s)=\{(a,\phi (a))\mid a\in \mathbb{Z}^X\}$ be the group of
$I$-type associated to a finite solution $(X,s)$. We know that this
is a left brace as explained in the comment before
Proposition~\ref{Itype} and thus, for $(a,\phi (a)),\, (b,\phi
(b))\in G(X,s)$, we have
\begin{equation} \label{needed-equation}
\lambda_{(a,\phi (a))}(b,\phi (b))=(a,\phi (a))(b,\phi (b))-(a,\phi (a))=
(\phi (a)(b),\phi (\phi (a)(b))). \end{equation} Thus
$\lambda_{(a,\phi(a))}=\id_{G(X,s)}$ if and only if
$\phi(a)=\id$.

The socle of the brace $G(X,s)$ is
\begin{eqnarray*}
\Soc (G(X,s))&=&\{ (a,\phi(a))\in G(X,s)\mid \lambda_{(a,\phi(a))}=\id\}\\
 &=&\{ (a,\phi(a))\in G\mid \phi(a)=\id\}.
 \end{eqnarray*}
Note that the multiplicative group of the left brace $\bar
G=G(X,s)/\Soc(G(X,s))$ is isomorphic to the IYB group
$\mathcal{G}(X,s)$ associated to the solution $(X,s)$. We denote by
$\bar a$ the element $\bar a=\Soc(G(X,s))(a,\phi(a))\in \bar G$. Let
$X_1=\{ \bar x\mid x\in X\}\subseteq \bar G$. Then $X_1$ generates
the additive group of the brace $\bar G$. Furthermore, for $x,y\in
X$, we get that $\bar x=\bar y$ if and only if $\phi(x)=\phi(y)$.
Let $x,z,y,t\in X$ be such that $\bar x=\bar z$ and $\bar y=\bar t$.
We have that $\phi(x)=\phi(z)$. Since $\bar y=\bar t$, this implies
that $(y,\phi(y))-(t,\phi(t))\in \Soc(G(X,s))$. Hence
$\lambda_{(x,\phi(x))}(y-t,\id)=\lambda_{(x,\phi(x))}(y,\phi(y))-\lambda_{(x,\phi(x))}(t,\phi(t))\in\Soc(G(X,s))$.
Therefore
$\overline{\phi(x)(y)}=\overline{\phi(x)(t)}=\overline{\phi(z)(t)}$.
Then $\phi(\phi(x)(y))^{-1}=\phi(\phi(z)(t))^{-1}$. Since $\bar
x=\bar z$, we also obtain that
$\overline{\phi(\phi(x)(y))^{-1}(x)}=\overline{\phi(\phi(z)(t))^{-1}(z)}$.
 Hence the mapping
 $s_1\colon X_1^2\longrightarrow X_1^2$ defined by
$$s_1(\bar x,\bar y)=(\overline{\phi(x)(y)},\overline{\phi(\phi(x)(y))^{-1}(x)}),$$
 for $x,y\in X$, is well defined, and it is easy to see that $(X_1,s_1)$ is a
 non-degenerate involutive set-theoretic solution of the Yang-Baxter equation.
 Let $\Ret (X,s)= (X/\sim , \widetilde{s})$. We denote by $[a]$ the
retraction class of $a\in X$, i.e. $[a]=\{ b\in X \mid \phi (a)=\phi
(b)\}$. It is easy to see that the map $X_{1}\rightarrow X/\sim$
defined by $\bar{a} \mapsto [a]$ gives an isomorphism from    the
solution $(X_1,s_1)$ to the solution  $\Ret (X,s)=(X/\sim ,
\widetilde{s})$.

Since the  multiplicative group of the left brace $\bar{G}$ is
abelian, it is a two-sided brace. By Proposition~\ref{radical}, the
solution $(\bar{G}, \bar{r})$ associated to the finite two-sided
brace $\bar{G}$ is a multipermutation solution. Note that by the
definition of $s_{1}$ and by (\ref{needed-equation}) one obtains
that  $s_{1}$ is the restriction of $\bar{r}$ to $X_{1}^{2}$. Since
$(\bar{G},\bar{r})$ is a  multipermutation solution,
Lemma~\ref{reduction2} implies that $(X_{1},s_{1})$ also is a
multipermutation solution. Therefore $(X,s)$ is a multipemutation
solution.
\end{proof}

Note the first part of the proof is general. The abelian hypothesis
of the associated IYB group $\mathcal{G}(X,s)$ is only used in the
last paragraph in order to deduce that $(\bar{G},\bar{r})$ is a
two-sided brace. Proposition~\ref{radical} then yields it is a
multipermutation solution.

\section{On some open problems} \label{openproblems}

We say that a non-degenerate involutive set-theoretic solution
$(X,s)$ of the Yang-Baxter equation is square free if $s(x,x)=(x,x)$
for all $x\in X$. Such solutions have received a lot of attention as
they seem to have better structural properties (see for example
\cite{GC,JO,JObook} and their references). In particular Rump proved
in \cite{rump1} that such a finite non-trivial square-free solution
is  decomposable in the following sense. The set $X$ is the disjoint
union of non-empty $\mathcal{G}(X,s)$-invariant subsets $X_{1}$ and
$X_{2}$. Gateva-Ivanova in \cite{Gat} used Rump's result to prove
that such solutions come from the so called semigroups of skew
polynomial type, hence confirming her earlier conjecture stated in
\cite{Gat-Mis}.

The following conjecture was formulated by Gateva-Ivanova in
\cite{Gat}.

\begin{conjecture}[Gateva-Ivanova] \label{Conjecture-Gateva}
Every set-theoretic non-degenerate involutive square-free solution
$(X,s)$ of the Yang-Baxter equation of cardinality $n\geq 2$ is a
multipermutation solution of level $m<n$.
\end{conjecture}

Recently, in \cite[Open questions 6.13]{GC}, Cameron and
Gateva-Ivanova stated the following two questions.

\begin{question}
Let $(X,s)$ be a finite multipermutation square-free solution of
the Yang-Baxter equation with $|X|>1$ and $\mpl(X,s)=m$.
\begin{enumerate}
\item Can we find a lower bound for the solvable length of the
group of $I$-type associated to $(X,s)$ in terms of $m$?
\item Are there multipermutation square-free solutions $(X,s)$
of arbitrarily  high multipermutation  level with an  abelian $IYB$
group $\mathcal{G}(X,s)$? If not, what is the largest integer $M$
for which there exist solutions $(X,s)$ with $\mpl(X,s)=M$ and
$\mathcal{G}(X,s)$ abelian?
\end{enumerate}
\end{question}

Rump in \cite{rump4} constructed finite multipermutation
solutions of the Yang-Baxter equation of arbitrary
multipermutation level such that their associated IYB groups are
abelian. But these solutions are not square free.

The following result provides an answer to both questions.  Let $G$
be a left brace. Let $(G,r)$ be its associated non-degenerate
involutive solution of the Yang-Baxter equation, that is,
$r(a,b)=(\lambda_a(b),\lambda^{-1}_{\lambda_a(b)}(a))$ for $a,b\in
G$. Note that $(G,r)$ is square free if and only if $\lambda_a(a)=a$
for all $a\in G$. We have the following result.

\begin{theorem} \label{elementary-abelian} Let $n$ be a positive integer.
Then there exists a finite multipermutation square-free solution
of the Yang-Baxter equation of multipermutation level $n$  such
that its associated IYB group is an elementary abelian $2$-group.
\end{theorem}

\begin{proof}
Let $G$ be an elementary abelian $2$-group of order $2^n$. Let
$R=\omega((\mathbb{Z}/2\mathbb{Z})[G])$, the augmentation ideal of
the  group ring $(\mathbb{Z}/2\mathbb{Z})[G]$. Note that if
$G=\langle g_1,\dots, g_n\rangle$, then $\sum_{g\in
G}g=\prod_{i=1}^n(g_i+1)\in R^n$ and $R^{n+1}=\{ 0\}$. Clearly we
have  $R^{m+1}\subseteq R^{m}$ and $R^{m+1}\neq R^{m}$ for all
positive integers $m\leq n$. From Proposition~\ref{rad} we  know
that $(R,+,\circ)$ is a two-sided brace. Furthermore,
$(R/R^m,+,\circ)$ also is a two-sided brace for all positive
integers $m\leq n$. Consider the set-theoretic solutions $r\colon
R\times R\longrightarrow R\times R$ and $r_{n+1-m}\colon R/R^m\times
R/R^m \longrightarrow R/R^m\times R/R^m$ associated to these braces.
Recall that
\begin{eqnarray*}r(\alpha,\beta)&=&(\lambda_{\alpha}(\beta),\lambda^{-1}_{\lambda_{\alpha}(\beta)}(\alpha)),\\
r_{n+1-m}(\overline{\alpha},\overline{\beta})&=&(\lambda_{\overline{\alpha}}(\overline{\beta}),
\lambda^{-1}_{\lambda_{\overline{\alpha}}(\overline{\beta})}(\overline{\alpha})),
\end{eqnarray*}
for all $\alpha,\beta\in R$ and
$\overline{\alpha},\overline{\beta}\in R/R^{m}$.  An easy
calculation yields
\begin{eqnarray*}r(\alpha,\beta)&=&(\alpha\beta+\beta,\alpha\beta+\alpha),\\
r_{n+1-m}(\overline{\alpha},\overline{\beta})&=&(\overline{\alpha}\overline{\beta}+\overline{\beta},
\overline{\alpha}\overline{\beta}+\overline{\alpha}).
\end{eqnarray*}

Note that $\lambda_{\alpha}(\alpha)=\alpha^2+\alpha=\alpha$ for all
$\alpha\in R$. Thus $(R,r)$ and $(R/R^m,r_{n+1-m})$ are square-free
solutions. By Lemma~\ref{lambda},
$\lambda_{\alpha}\lambda_{\beta}=\lambda_{\alpha\circ\beta}$.
Therefore the $IYB$ group $G_r=\langle \lambda_{\alpha}\mid\alpha\in
R\rangle$ associated to $(R,r)$ is abelian and, since
$\alpha\circ\alpha=\alpha^2+\alpha+\alpha=0$ and $\lambda_0=\id_R$,
we have that $G_r$ is an elementary abelian $2$-group.

Note also that $\Soc(R)=\{\alpha\in R\mid \alpha\beta=0$ for all
$\beta\in R\}$ and $\Soc(R/R^m)=\{\overline{\alpha}\in R/R^m\mid
\overline{\alpha}\overline{\beta}=0$ for all $\overline{\beta}\in
R/R^m\}$. Therefore $\Soc(R)=R^n$ and $\Soc(R/R^m)=R^{m-1}/R^m$
for all $1<m\leq n$. Thus, by Lemma~\ref{brace-soc}, we have that $\Ret(R,r)=(R/R^n,r_1)$
and, since $(R/R^m)/(R^{m-1}/R^{m})\cong R/R^{m-1}$,
$\Ret(R/R^m,r_{n+1-m})=(R/R^{m-1},r_{n+2-m})$. Hence
$\Ret^m(R,r)=(R/R^{n+1-m},r_m)$ for all positive integers $m\leq
n$. Therefore $(R,r)$ is a multipermutation solution of
multipermutation level $n$, and the result follows.
\end{proof}

Assertion (i) in the following result shows that examples of
solutions as in the proof of Theorem~\ref{elementary-abelian} cannot
come from a left brace $G$ whose order is odd. Furthermore, the
result gives  a positive answer to
Conjecture~\ref{Conjecture-Gateva} in the case of solutions of the
Yang-Baxter equation associated to finite left braces.

\begin{theorem}\label{square}
Let $G$ be a left brace such that $\lambda_a(a)=a$ for all $a\in G$.
Then $G$ is a two-sided brace, the multiplicative group of $G$ is
nilpotent of class at most $2$ and $\Soc(G)\subseteq \Z(G)$,  the
center of $G$.
\begin{itemize}
\item[(i)] The torsion-subgroup $T(G)$ of the multiplicative group of $G$
is an ideal of $G$. If $G_e$ is the unique maximal $2$-subgroup of
$T(G)$ and $G_{o}$ is the maximal subgroup of $T(G)$ consisting of
the elements of odd order, then both $G_{e}$ and $G_{o}$ are
ideals of the brace $G$ and $\Soc(T(G))=\Soc(G_{e})\Z(G_{o})$.
\item[(ii)] Furthermore, if $G_e$ is finite and the multiplicative group of $G/T(G)$ is
finitely generated, then  the
solution of the Yang-Baxter equation  $(G,r)$ associated to the
brace $G$ is a multipermutation solution.
\end{itemize}
\end{theorem}

\begin{proof}
Let $a\in G$. We  claim that $a^n=na$ for every positive integer $n$
by induction on $n$.  We prove this by induction on $n$. For $n=1$
this is clear. For $n=2$ we have $a^2=\lambda_a(a)+a=a+a=2a$.
Suppose that $n>2$ and that $a^m=ma$ for all $1\leq m<n$. Now we
have
\begin{eqnarray*}
a^n&=&aa^{n-1}=a((n-1)a)=a((n-2)a+a)\\
&=&a(a^{n-2}+a)=a^{n-1}+a^2-a=(n-1)a+2a-a=na. \end{eqnarray*} This
proves the claim. As a consequence we obtain that the multiplicative
order of $a$ is equal to its additive order.

If $a,b,c\in G$ then $a(c-b)+ab=a((c-b)+b)+a= ac+a$. Hence
$a(c-b)=ac-ab+a$ and in particular $a(1-a)=a-a^2+a$. Therefore
$a(-a)=a(1-a)=a-a^2+a=a-(a+a)+a=0$ and $a^{-1}=-a$. Thus, we also
get
\begin{eqnarray*}(b+c)a+a&=&(a^{-1}(b+c)^{-1})^{-1}+a\\
&=&(a^{-1}(-b-c))^{-1}+a=(a^{-1}(b^{-1}+c^{-1}))^{-1}+a\\
&=&(a^{-1}b^{-1}+a^{-1}c^{-1}-a^{-1})^{-1}+a\\
&=&((ba)^{-1}+(ca)^{-1}+a)^{-1}+a\\
&=&(-(ba)-(ca)+a)^{-1}+a=ba+ca-a+a\\
&=&ba+ca
\end{eqnarray*}
Hence $G$  also is a right brace and thus $G$ is a two-sided brace.

Since $G$ is a two-sided brace, $R=(G,+,*)$ is a radical ring, where
$a*b=ab-a-b$ for $a,b\in G$, see Proposition~\ref{rad}. The
condition $\lambda_{a}(a)=a$ is equivalent to saying that $a*a=1$,
which is the zero element of this radical ring. It follows that $R$
is anticommutative. Hence, for $a,b,c\in G$, we get
$a*b*c=-(a*c*b)=c*a*b$. Therefore $G*G$ lies in the center of $R$
and thus  the  multiplicative group of the brace $G$, which is the
adjoint group of the ring $R$, is nilpotent of class at most $2$.
Let $a\in \Soc(G)$. Then, for every $b\in G$, we have
$ab=\lambda_a(b)+a=b+a=-b^{-1}-a^{-1}=-\lambda_{a^{-1}}(b^{-1})
-a^{-1}=-a^{-1}b^{-1}=-(ba)^{-1}=ba$. Therefore $\Soc(G)\subseteq
\Z(G)$.

$(i)$ Since the multiplicative group of $G$ is nilpotent, by
\cite[5.2.7]{Robinson} the elements of finite order in $G$ form a
fully-invariant subgroup $T(G)$, $G/T(G)$ is torsion-free and
$T(G)=G_{e}G_{o}$, a direct product. We have seen that the additive
order of every $c\in G$ is equal to its multiplicative order. Hence,
by Lemma~\ref{lambda},  $T(G),G_{o},G_{e}$ are ideals of the
two-sided brace $G$ and they are braces themselves.

Let $a,b\in G$. Assume that $ab=ba$. Then
\begin{eqnarray*}
a+b&=&\lambda_{a+b}(a+b)=(a+b)(a+b)-a-b\\
&=&(a+b)a+(a+b)b-2a-2b\\
&=&a^2+ba-a+ab+b^2-b-2a-2b\\
&=&\lambda_a(a)+\lambda_b(b)+2ab-2a-2b=2ab-a-b.
\end{eqnarray*}
Thus $2(a+b)=2ab$.

Since the order of every element of $G_{o}$ is odd, it now follows
that $a+b=ab$, for all $a\in \Z(G_{o}),b\in G_{o}$. Therefore
$\lambda_a(b)=ab-a=b$ for all $a\in \Z(G_{o}),b\in G_{o}$. This
implies that $\Z(G_{o})\subseteq \Soc(G_{o})$. Since we also have
$\Z(G_{o})\supseteq \Soc(G_{o})$, we get $\Soc (G_{o})=\Z(G_{o})$.

For $a\in G_{o}, b\in G_{e}$ we have $\lambda_{a}(b) =ab -a \in
G_{e}, \lambda_{b}(a)=ba - b \in G_{o}$.    In view of
Lemma~\ref{lambda}, the elements $a,\lambda_{a}(b)$ are the
components of $ab$ in $(T(G),+)=G_{o}\oplus G_{e}$. Also,
$\lambda_{b}(a),b$ are the components of $ba$ in  $G_{o}\oplus
G_{e}$. Since $ab=ba$, we get that  $\lambda_{a}(b)=b$ and
$\lambda_{b}(a)=a$. Therefore $\Z(G_{o})\subseteq \Soc (T(G))$ and
$\Soc(G_{e})\subseteq \Soc(T(G))$. Thus
$\Z(G_{e})\Soc(G_{o})\subseteq\Soc(T(G))$. Conversely, let $c\in
\Soc(T(G))$.
 Then $c=ab=ba$ for $a\in G_{o}, b\in G_{e}$.
 Then $\lambda_{a}\lambda_{b}=\lambda_{b}\lambda_{a}=\id$.
 Now, for $x\in G_{o},y\in G_{e}$ we have
 $$x+y=\lambda_{a}\lambda_{b}(x+y)=\lambda_{a}\lambda_{b}(x)+\lambda_{a}\lambda_{b}(y)=
\lambda_{a}(x)+\lambda_{b}(y).$$ Hence $\lambda_{a}(x)=x,
\lambda_{b}(y)=y$.  Therefore $a\in \Soc(G_{o})$ and $b\in
\Soc(G_{e})$ and $(i)$ follows.

$(ii)$  We know that $G*G$ lies in the center of the radical ring
$R$. Since $R$ is anticommutative,  $2G*G*G=0$. Therefore
$0=a*b=ab-a-b=\lambda_a(b)-b$, for all $a\in 2G*G$ and $b\in G$.
Hence  $2G*G\subseteq \Soc(G)$. Let $G_1=G/\Soc(G)$  be the quotient
brace. Recall that the ideals of the radical ring $R$ coincide with
the ideals of the two-sided brace $G$ and if $J$ is an ideal of $R$,
then the ring $R/J$ is the radical ring corresponding to the
quotient brace of $G$ modulo $I$ (see the end of
Section~\ref{definitions}). Then clearly
$(2G+\Soc(G))/\Soc(G)\subseteq\Soc(G_1)$. Note that
$T(G)+2G+\Soc(G)=G_e+2G+\Soc(G)$. Let $I=G_e+2G+\Soc(G)$.  By
hypothesis, the multiplicative group of $G/T(G)$ is finitely
generated. Therefore also the multiplicative group of $G/I$ is
finitely generated (note that $G/I$ is the quotient brace, so it
also is a homomorphic image of the multiplicative group of $G$). As
in the beginning of the proof, the additive orders and the
multiplicative orders of elements in $G/I$ are the same. Whence the
multiplicative group of $G/I$ has the property that every element is
of order at most 2. So it is an elementary finite 2-group. Thus
$|G_1/\Soc(G_1)|\leq |G/I|\cdot |I/(2G+\Soc(G))|\leq|G/I|\cdot
|G_e|<\infty$. By Lemma~\ref{brace-soc}, $\Ret^2(G,r)$ is the
solution of the Yang-Baxter equation associated to the two-sided
brace $G_1/\Soc(G_1)$. Therefore, from Proposition~\ref{radical}, we
get that $(G,r)$ is a multipermutation solution.
\end{proof}

In the special case where $G=G_{o}$  and $G$ is abelian we get from
Theorem~\ref{square} that  $G=\Soc(G)$ and the associated solution
of the Yang-Baxter equation is a multipermutation solution of level
$\leq 1$ (compare with Theorem~\ref{elementary-abelian}).

In Section~\ref{Section-Examples} we give  some examples of
two-sided braces with square-free associated solutions and whose
multiplicative groups are not abelian. In particular, examples that
are nilpotent of class two exist.

Let $(X,s)$ be a finite set-theoretic non-degenerate involutive
square-free solution of the Yang-Baxter equation. Let $\bar
G=G(X,s)/\Soc(G(X,s))$. From the remark given after  the proof of
Theorem~\ref{abelian} we know that if the solution
$(\bar{G},\bar{r})$ associated to the left brace $\bar G$ is square
free then by Theorem~\ref{square} it is a multipermutation solution
and  thus $(X,s)$ is a multipermutation solution as well. Also from
the proof of Theorem~\ref{abelian}  we know that $\mathcal{G}(X,s)$
is isomorphic to the multiplicative group of $\bar{G}$. Hence
$\mathcal{G}(X,s)$ inherits a natural left brace structure from the
left brace $\bar{G}$. Therefore the solution of the Yang-Baxter
equation associated to the left brace $\mathcal{G}(X,s)$ is
isomorphic to $(\bar{G},\bar{r})$. However the following example
shows that $(\bar{G},\bar{r})$ is not necessarily square free. Using
Proposition~\ref{solution} it is easy to verify that
$(X=\{1,2,3,4,5,6\},s)$ is a  solution of the Yang-Baxter equation,
where $s(i,j)=(\sigma_i(j),\sigma^{-1}_{\sigma_i(j)}(i))$, with
$\sigma_1=\sigma_2=\sigma_3=\id$, $\sigma_4=(1,2,3)=\sigma_5^{-1}$
and $\sigma_6=(4,5)(2,3)$. Clearly it is square free, but its  $IYB$
group is isomorphic to $\Sym_3$ and it is not nilpotent. Thus by
Theorem~\ref{square}, the solution of the Yang-Baxter equation
associated to the left brace $\mathcal{G}(X,s)$ is not square free.

\section{Relations between braces and group rings} \label{groupringrelations}

It is easy to see that every abelian group is the multiplicative
group of a two-sided brace (see Section~\ref{Section-Examples}), in
particular finite abelian groups are IYB groups by
Corollary~\ref{finite-group-IYB}. Moreover,
Conjecture~\ref{Conjecture-Gateva} has a positive solution in case
the associated IYB group is abelian (Theorem~\ref{abelian} and
\cite[Theorem~7.1]{GC}). Thus a next natural step is to consider the
solutions for which the associated IYB group is nilpotent.

Natural examples of finite nilpotent groups are the adjoint groups
of finite  radical rings. However not every finite nilpotent group
is of this type, or equivalently it is  not the multiplicative group
of a finite two-sided brace (Proposition~\ref{rad}). This easily
follows from the following result of Kruse \cite{kruse}. If a finite
$p$-group $G$ is the adjoint group of a nilpotent ring, then $G$ has
a central series $G=Z_{0} \supseteq  Z_{1} \supseteq \cdots
\supseteq Z_{c}=\{ 1 \}$ in which $[Z_{i-1}: Z_{i}]\geq p^{2}$ for
$1\leq i< c$.
 However, we easily can prove the following result.

\begin{lemma}\label{pgroup}
Let $G$ by a locally finite $p$-group. Then $G$ is isomorphic to a
subgroup of the adjoint group of a radical ring. Furthermore, if $G$
is a finite $p$-group, then it is isomorphic to a subgroup of the
adjoint group of a radical ring $R$ such that $|R|$ is a power of
$p$.
\end{lemma}
\begin{proof}
Let $K=\mathbb{Z}/p\mathbb{Z}$ be the field of $p$ elements. By
\cite[Lemma 8.1.17]{passman2}, $\omega(K[G])=J(K[G])$ is a radical
ring. Let $f\colon G\longrightarrow \omega(K[G])$ be the map defined
by $f(g)=g-1$ for all $g\in G$. It is clear that $f$ is injective.
Let $g,h\in G$. Then
$$f(gh)=gh-1=(g-1)(h-1)+(g-1)+(h-1)=(g-1)\circ
(h-1)=f(g)\circ f(h),$$ and the result follows.
\end{proof}

After these rather elementary comments we now give a recent
intriguing connection between the involutive Yang-Baxter groups (or
more general multiplicative groups of left braces) and fundamental
problems on integral group rings.

\begin{proposition} (Sysak \cite{sysak, sysak2})
Let $G$ be a group. Then $G$ is the multiplicative group of a left
brace if and only if there exists a left ideal $L$ of
$\mathbb{Z}[G]$ such that
\begin{itemize}
\item[(i)] the augmentation ideal $\omega(\mathbb{Z}[G])=G-1+L$ and
\item[(ii)] $G\cap (1+L)=\{ 1\}$.
\end{itemize}
\end{proposition}

\begin{proof}
Suppose that there exists a left ideal $L$ of $\mathbb{Z}[G]$
satisfying $(i)$ and $(ii)$. Let $a,b\in G$. Then by $(i)$ there
exist $c\in G$ and $\alpha\in L$ such that $a-1+b-1=c-1+\alpha$.
Furthermore, if $a-1+b-1=d-1+\beta$, for $d\in G$ and $\beta\in L$,
then $c-d\in L$. Hence $d^{-1}c-1\in L$ and, by $(ii)$, $d^{-1}c=1$.
Thus there is a unique $c\in G$ such that $a-1+b-1-c+1\in L$. We
define a sum $\oplus$ on $G$ by
$$(a\oplus b)-a-b+1\in L,$$
for $a,b\in G$.

Clearly $a\oplus b=b\oplus a$. Let $a,b,c\in G$. Then $((a\oplus
b)\oplus c)-(a\oplus b)-c+1\in L.$ Since $(a\oplus b)-a-b+1\in L$,
we have that $((a\oplus b)\oplus c)-a-b-c+2\in L.$ Since $(b\oplus
c)-b-c+1\in L$, we have $((a\oplus b)\oplus c)-a-(b\oplus c)+1\in
L.$ Therefore $(a\oplus b)\oplus c=a\oplus (b\oplus c).$ Note that
$(a\oplus 1)-a-1+1\in L$. Therefore $a^{-1}(a\oplus 1)\in 1+L$, and
by $(ii)$, $a\oplus 1=a$. Furthermore, for  $a\in G$, there exists a
unique $\ominus a\in G$ such that $1-a-(\ominus a)+1\in L$. Then
$a\oplus(\ominus a)=1$. Hence $(G,\oplus)$ is an abelian group. Let
$a,b,c\in G$. Then
\begin{eqnarray*} (c(a\oplus
b)\oplus c)-ca-cb+1&=&(c(a\oplus b)\oplus c)-c(a\oplus
b)-c+1\\
&&+\, c(a\oplus
b)+c-1-ca-cb+1\\
&=&(c(a\oplus b)\oplus c)-c(a\oplus b)-c+1\\
&&+\, c((a\oplus b)-a-b+1)\in L.
\end{eqnarray*}
Therefore $c(a\oplus b)\oplus c=ca\oplus cb.$ Thus $G$ is the
multiplicative group of a left brace.

Suppose now that $G$ is the multiplicative group of a left brace.
Denote by $\oplus$ the sum of the left brace $G$. Let $I$ be the
additive subgroup of $\omega(\mathbb{Z}[G])$ generated by all the
elements of the form $(a\oplus b)-a-b+1$, for $a,b\in G$. We shall
see that $I$ is a left ideal of $\mathbb{Z}[G]$. To prove this, it
is sufficient to show that $c((a\oplus b)-a-b+1)\in I$, for all
$a,b,c\in G$. Note that
\begin{eqnarray*}
c((a\oplus b)-a-b+1)&=&c(a\oplus b)-ca-cb+c\\
&=&(c(a\oplus b)\oplus c)-(c(a\oplus b)\oplus c)\\
&&+\, c(a\oplus b)-ca-cb+c\\
&=&(ca\oplus cb)-(c(a\oplus b)\oplus c)+c(a\oplus b)-ca-cb+c\\
&=&(ca\oplus cb)-ca-cb+1\\
&&-\, (c(a\oplus b)\oplus c)+c(a\oplus b)+c-1 \in I.
\end{eqnarray*}
Therefore $I$ is a left ideal.

Let $\alpha=\sum_{x\in G}\alpha_xx\in\omega( \mathbb{Z}[G])$. We
shall prove that $\alpha\in G-1+I$ by induction on $\sum_{x\in
G}|\alpha_x|$. If $\sum_{x\in G}|\alpha_x|=0$, then $\alpha=0=1-1\in
G-1+I$. If $\sum_{x\in G}|\alpha_x|=2$, then $\alpha=x-y$, for some
$x,y\in G$ with $x\neq y$. In this case, let $z\in G$ such that
$y\oplus z=1$. Then
\begin{eqnarray*}\alpha&=&x-1+(y\oplus
z)-y-z+1+z-1\\
&=&(x\oplus z)-1-(x\oplus z)+x+z-1+(y\oplus z)-y-z+1\in
G-1+I.\end{eqnarray*} Suppose that $\sum_{x\in G}|\alpha_x|=n>2$ and
that $\beta\in G-1+I$, for all $\beta=\sum_{x\in G}\beta_xx\in
\omega(\mathbb{Z}[G])$ such that $\sum_{x\in G}|\beta_x|<n$. Because
$\alpha \in \omega (\mathbb{Z}G)$ there exist $x,y\in G$ such that
 $\alpha_x\alpha_y<0$.  By symmetry we may assume that
 $\alpha_x<0$.
 Then, $\beta=\alpha+x-y\in G-1+I$, by induction
hypothesis, and therefore $\alpha=y-x+g-1+\gamma$, for some $g\in G$
and some $\gamma\in I$. Since $y-x=h-1+\delta$ for some $h\in G$ and
some $\delta\in I$, we have that
$$\alpha=h+g-2+\gamma+\delta=(h\oplus g)-1-(h\oplus g)+h+g-1+\gamma+\delta\in G-1+I.$$
Hence $\omega(\mathbb{Z}[G])=G-1+I$. Note that
$\omega(\mathbb{Z}[G])$ is the free abelian group $\bigoplus_{g\in
G}\mathbb{Z}(g-1)$. Let
$\psi\colon\omega(\mathbb{Z}[G])\longrightarrow (G,\oplus)$ be the
morphism of abelian groups such that $\psi(g-1)=g$. It is easy to
see that $I\subseteq \ker(\psi)$. Therefore $G\cap (1+I)=\{ 1\}$.
\end{proof}

In a similar way we can prove the following result.

\begin{proposition}\label{twosided}
Let $G$ be a group. Then $G$ is the multiplicative group of a
two-sided brace if and only if there exists an ideal $L$ of
$\mathbb{Z}[G]$ such that
\begin{itemize}
\item[(i)] the augmentation ideal $\omega(\mathbb{Z}[G])=G-1+L$ and
\item[(ii)] $G\cap (1+L)=\{ 1\}$.
\end{itemize}
\end{proposition}

Proposition~\ref{twosided} for finite groups essentially was proved
by Sandling \cite[Theorem 1.5]{sandling}. It is a crucial property
to prove the following positive solution to the isomorphism problem
of integral group rings of adjoint groups of finite radical rings.
(See also \cite[Section 9.4]{PolcinoMiliesSehgal} for a proof of
this result.)

\begin{theorem}  (Sandling, \cite[Theorem 3.1]{sandling})
Let $G$ be a finite group isomorphic to the adjoint group of a
radical ring. If  $H$ is  a group such that $\mathbb{Z}[G]\cong
\mathbb{Z}[H]$, then $G\cong H$.
\end{theorem}

Roggenkamp and Scott gave in \cite{RS} an affirmative answer to the
isomorphism problem for finite nilpotent groups, thus generalizing
the result of Sandling.

\begin{remark}   \label{zg} {\rm
Suppose that for a group $G$, there exists a left ideal $L$ of
$\mathbb{Z}[G]$ such that $\omega(\mathbb{Z}[G])=G-1+L$ and $G\cap
(1+L)=\{ 1\}$. Let $\alpha$ be a unit in $\mathbb{Z}[G]$ of
augmentation $1$. Then there exist $g\in G$ and $\beta\in L$ such
that
$$\alpha-1=g-1+\beta.$$
Thus $\alpha=g(1+g^{-1}\beta)$. Thus the group of units of
$\mathbb{Z}[G]$ is $\mathcal{U}(\mathbb{Z}[G])=(\pm G)H$, where
$H=(1+L)\cap \mathcal{U}(\mathbb{Z}[G])$. Furthermore, since $G\cap
(1+L)=\{ 1\}$, we have that $(\pm G)\cap H=\{ 1\}$. If $L$ is a
two-sided ideal, then $H$ is a normal subgroup of
$\mathcal{U}(\mathbb{Z}[G])$. Such normal subgroups $H$ are called
normal complements in the group of units of the group ring. The
existence of normal complements is related with the isomorphism
problem. For finite nilpotent groups, an affirmative answer to the
existence of normal complements implies an affirmative answer to the
isomorphism problem \cite[Proposition (30.4)]{sehgal}. The existence
of normal complements is known for  finite nilpotent groups of class
$2$. But for general finite nilpotent groups it remains an open
question (see \cite[Section 34 and Problem 30]{sehgal}). }
\end{remark}

\section{Constructions and examples of braces} \label{Section-Examples}

The results stated in the previous sections show that an important
problem is to describe when a (finite) group is the multiplicative
group of a left brace. In this section we give several examples of
such groups. We begin with a natural construction that can be
applied in this context.

Let $G,H$ be two left braces. Then it is easy to see that the direct
product $G\times H$ of the multiplicative groups of the braces $G$
and $H$ with the sum defined componentwise is a left brace called
the direct product of the braces $G$ and $H$.

We shall construct the semidirect product of two braces as
introduced by Rump in \cite{rump6}. Let $N,H$ be left braces. Let
$\eta\colon H\longrightarrow \Aut(N)$ be a homomorphism of groups
from the multiplicative group of $H$ to the group of automorphisms
of the left brace $N$ (see Definition~\ref{hom}).

Consider the semidirect product $N\rtimes H$ of the multiplicative
groups $N$ and $H$ via $\eta$. Define in $N\rtimes H$ a sum by
$$(g_1,h_1)+(g_2,h_2)=(g_1+g_2,h_1+h_2),$$
for $g_1,g_2\in N$ and $h_1,h_2\in H$. It is clear that $(N\rtimes
H, +)$ is an abelian group. We prove that the group $N\rtimes H$
with this sum is a left brace. Indeed, let $g_1,g_2,g_3\in N$ and
$h_1,h_2,h_3\in H$. Then
\begin{eqnarray*}
\lefteqn{(g_1,h_1)((g_2,h_2)+(g_3,h_3))+(g_1,h_1)}\\
&=&(g_1,h_1)(g_2+g_3,h_2+h_3)+(g_1,h_1)\\
&=&(g_1\eta(h_1)(g_2+g_3),h_1(h_2+h_3))+(g_1,h_1)\\
&=&(g_1(\eta(h_1)(g_2)+\eta(h_1)(g_3)),h_1(h_2+h_3))+(g_1,h_1)\\
&=&(g_1(\eta(h_1)(g_2)+\eta(h_1)(g_3))+g_1,h_1(h_2+h_3)+h_1)\\
&=&(g_1\eta(h_1)(g_2)+g_1\eta(h_1)(g_3),h_1h_2+h_1h_3)\\
&=&(g_1\eta(h_1)(g_2),h_1h_2)+(g_1\eta(h_1)(g_3),h_1h_3)\\
&=&(g_1,h_1)(g_2,h_2)+(g_1,h_1)(g_3,h_3).\\
\end{eqnarray*}
Hence, $N\rtimes H$ is a left brace. We call this left brace the
semidirect product of the left braces $N$ and $H$.

Similarly the semidirect product $A\ltimes B$ of two right braces
$A,B$ is defined (via a right action of the right brace $A$ on the
right brace $B$).

In view of Remark~\ref{braceIYB}, this construction
of the semidirect product  for finite braces is
essentially \cite[Theorem 3.4]{CJR}. As a consequence we can prove
the following result that generalizes \cite[Corollary 3.5 and Corollary 3.6]{CJR}.

\begin{corollary}
Let $G,H$ be two left braces. Then the wreath product $G\wr H$ of
the groups $G$ and $H$ is the multiplicative group of a left brace.
Moreover, every finite solvable group is a subgroup of an IYB group.
\end{corollary}

\begin{proof}
Recall that the wreath product of $G$ by $H$ is the semidirect
product $W\rtimes H$, where $W=\{ f\colon H\longrightarrow G\mid
|\{h\in H\mid f(h)\neq 1\}|<\infty\}$ and the action of $H$ on $W$
is given by the homomorphism $\sigma\colon H\longrightarrow \Aut(W)$
defined by $\sigma(h)(f)(x)=f(hx)$, for all $h,x\in H$ and $f\in W$.

We define a sum on $W$ by $(f+g)(h)=f(h)+g(h)$, for $f,g\in W$ and
$h\in H$. It is clear that $(W,+)$ is an abelian group. We shall see
that $W$ with this sum is a left brace. Let $f,g_1,g_2\in W$ and
$h\in H$. Then
\begin{eqnarray*}
(f(g_1+g_2)+f)(h)&=&(f(g_1+g_2))(h)+f(h)=f(h)(g_1+g_2)(h)+f(h)\\
&=&f(h)(g_1(h)+g_2(h))+f(h)=f(h)g_1(h)+f(h)g_2(h)\\
&=&(fg_1)(h)+(fg_2)(h)=(fg_1+fg_2)(h).
\end{eqnarray*}
Therefore $f(g_1+g_2)+f=fg_1+fg_2$, and thus $W$ is a left brace.

In order to prove that $G\wr H$ is the multiplicative group of a
left brace it is sufficient to show that for every $h\in H$ the map
$\sigma(h)$ is a homomorphism of left braces. Then $G\wr H$ will be
the multiplicative group of the semidirect product of the left
braces $W$ by $H$.

Let $h, x\in H$ and $f,g\in W$. Then
\begin{eqnarray*}
\sigma(h)(f+g)(x)&=&(f+g)(hx)=f(hx)+g(hx)=\sigma(h)(f)(x)+\sigma(h)(g)(x)\\
&=&(\sigma(h)(f)+\sigma(h)(g))(x).
\end{eqnarray*}
Therefore $\sigma(h)(f+g)=\sigma(h)(f)+\sigma(h)(g).$ Hence
$\sigma(h)$ is a homomorphism of left braces, and the first part of
the result follows.

The second part follows from the first part and  \cite[Satz
I.15.9]{Huppert}. The latter says that a finite solvable group is
a subgroup of an iterated wreath product of finite abelian groups.
\end{proof}

We continue with some concrete constructions of braces.

\begin{example} Abelian Groups\\
{\rm Let $G$ be a multiplicative abelian group. Define an operation
$+$ on $G$ by $$a+b=ab$$ for $a,b\in G$. Then clearly
$$a(b+c)+a=abca=abac=ab+ac$$ and thus  $$(b+c)a+a= ba+ca$$ for all
$a,b,c\in G$. Therefore $(G,+,\cdot)$ is a two-sided brace. In this
case $\lambda_{a}=\id_{G}$ for every $a\in G$. Further, $\Soc
(G)=G$. The solution of the Yang-Baxter equation associated to $G$
is a multipermutation solution of level at most $1$ with trivial IYB
group. Note that a multiplicative  abelian group admits many
structures of a brace (see for example the proof of
Theorem~\ref{elementary-abelian} and \cite{rump4}). }
\end{example}

\begin{example}  Nilpotent Groups  of class $2$  (Ault and Watters~\cite{aultwatters})\\
 {\rm  Let $G$ be a nilpotent group of class $2$. Suppose
that $\Z(G)$ contains a subgroup $Z$ so that $G/Z$ is a weak
direct product of cyclic groups. Let $[x]$ denote the image of
$x\in G$ under the natural homomorphism $G\longrightarrow G/Z$.
Then there exists a well-ordered set $(I,\leq)$ and $a_{i}\in G$,
for $i\in I$, such that $G/Z =\prod_{i\in I}\langle
\overline{a_{i}} \rangle $. Thus every element of $G$ can be
written in the form
$$a_{i_1}^{m_{1}}\cdots a_{i_n}^{m_{n}}z$$ for $i_1<\dots <i_n$ in $I$, some non-negative
integers $m_{1},\ldots ,m_{n}$ and some $z\in Z$. We define an
operation $+$ on $G$ by
$$a_{i_1}^{m_{1}}\cdots a_{i_n}^{m_{n}}z + a_{i_1}^{k_{1}}\cdots a_{i_n}^{k_{n}}z'
=a_{i_1}^{m_{1}+k_{1}}\cdots a_{i_n}^{m_{n}+k_{n}}zz'.$$ It is easy
to check that this operation is well-defined and $(G,+,\cdot)$ is a
two-sided brace. Clearly, $Z\subseteq \Soc (G)$. Hence, the solution
$(G,r)$ of the Yang-Baxter equation associated to $G$ is a
multipermutation solution of level  $2$ such that the group
$\mathcal{G}(G,r)$ is abelian. Note that there exist $i<j$ in $I$
such that $[a_j,a_i]\neq 1$. Since
$$\lambda_{a_{i}a_{j}}(a_{i}a_{j})=a_{i}a_{j}a_{i}a_{j}-a_{i}a_{j} =
a_{i}^{2}a_{j}^{2}[a_{j},a_{i}]- a_{i}a_{j}=
a_{i}a_{j}[a_{j},a_{i}]\neq a_{i}a_{j},$$ it follows that this
solution is not square free.}
\end{example}

In \cite{aultwatters} Ault and Watters conjectured that every
nilpotent group of class $2$ is the adjoint group of a radical ring
of nilpotency class $3$. They show that this is true for nilpotent
groups $G$ of class $2$  that satisfies one of the following
properties:
\begin{itemize}
\item[(i)] $G/\Z(G)$ is a weak direct product of cyclic groups, (as we have seen above).
\item[(ii)] $G/\Z(G)$ is a torsion group.
\item[(iii)] Every element of $[G,G]$ has a unique square root.
\end{itemize}
In \cite{HalesPassi} Hales and Passi gave a counterexample to the
conjecture of Ault and Watters. They also gave another proof of the
results of Ault and Watters in the cases $(i)$ and $(ii)$ and showed
that a nilpotent group $G$ of class $2$ is the adjoint group of a
radical ring of nilpotency class $3$ if it satisfies any of the
following properties:
\begin{itemize}
\item[(iv)] $G/\Z(G)$ is a uniquely $2$-divisible group.
\item[(v)] $G/N$ is a torsion-free group and completely decomposable
(i. e. a weak direct product of rank one groups) for some normal
subgroup $N$ such that $[G,G]\subseteq N\subseteq \Z(G)$.
\end{itemize}

The following result extends these results and gives many examples
of nilpotent groups of class $2$ with another structure of two-sided
brace and such that the solution of the Yang-Baxter equation
associated to this type of brace is square free. This is of interest
in view of the assertion of Theorem~\ref{square}.

\begin{proposition}\label{nilpotent}
Let $G$ be a nilpotent group of class $2$. Let $H$ be the subgroup
$H=\{ g^2z\mid g\in G,\; z\in \Z(G)\}$ of $G$. Then the operation
$+$ defined on $H$ by
$$h_1^2z_1+h_2^2z_2=(h_1h_2)^2z_1z_2[h_2,h_1]$$
for $h_1,h_2\in G$  and $z_1,z_2\in \Z(G)$, is well defined and
$(H,+,\cdot)$ is a two-sided brace.  The radical ring associated to
the two-sided brace $H$ is nilpotent of nilpotency class at most
$3$. Furthermore, the solution of the Yang-Baxter equation
associated to the brace $H$ is square free.
\end{proposition}

\begin{proof}
Let $h_1,h_2,g_1,g_2\in G$ and $z_1,z_2,z_1',z_2'\in \Z(G)$ be such
that $h_1^2z_1=g_1^2z_1'$ and $h_2^2z_2=g_2^2z_2'$. Note that
\begin{eqnarray*}
(h_1h_2)^2z_1z_2[h_2,h_1]&=&h_1h_2^2h_1z_1z_2=h_1g_2^2h_1z_1z_2'\\
&=&g_2h_1^2g_2z_1z_2'=g_2g_1^2g_2z_1'z_2'\\
&=&(g_1g_2)^2z_1'z_2'[g_2,g_1].
\end{eqnarray*}
Therefore the operation $+$ is well defined. It is easy to check
that $(H,+)$ is an abelian group. Let $h_3\in G$ and $z_3\in \Z(G)$.
Then
\begin{eqnarray*}
h_1^2z_1(h_2^2z_2+h_3^2z_3)+h_1^2z_1&=&h_1^2z_1(h_2h_3)^2z_2z_3[h_3,h_2]+h_1^2z_1\\
&=&(h_1h_2h_3)^2z_1z_2z_3[h_1,h_2][h_1,h_3][h_3,h_2]+h_1^2z_1\\
&=&(h_1h_2h_3h_1)^2z_1^2z_2z_3[h_1,h_2]^2[h_1,h_3]^2[h_3,h_2]\\
&=&(h_1^2h_2h_3)^2z_1^2z_2z_3[h_3,h_2]\mbox{ and }\\
h_1^2z_1h_2^2z_2+h_1^2z_1h_3^2z_3&=&(h_1h_2)^2z_1z_2[h_1,h_2]+(h_1h_3)^2z_1z_3[h_1,h_3]\\
&=&(h_1h_2h_1h_3)^2z_1^2z_2z_3[h_1,h_2]^2[h_3,h_2]\\
&=&(h_1^2h_2h_3)^2z_1^2z_2z_3[h_3,h_2].
\end{eqnarray*}
Therefore $(H,+,\cdot)$ is a left brace.
$$\lambda_{h_1^2z_1}(h_1^2z_1)=h_1^2z_1h_1^2z_1-h_1^2z_1=h_1^4z_1^2-h_1^2z_1=h_1^2z_1.$$
Hence the solution of the Yang-Baxter equation associated to the
brace $H$ is square free. By Theorem~\ref{square}, $H$ is a
two-sided brace. Consider the operation $*$ on $H$ defined by
$a*b=ab-a-b$ for all $a,b\in H$. It is easy to see that $a*b\in
\Z(G)$ for all $a,b\in H$. Note that $a*z=az-a-z=az-az=1$ for all
$a\in H$ and $z\in \Z(G)$. Therefore $H*H*H=\{ 1\}$, that is the
radical ring $(H,+,*)$ is nilpotent of nilpotency class at most $3$.
\end{proof}

Note that the definition of the sum in Proposition~\ref{nilpotent}
can also be written as
$h_1^2z_1+h_2^2z_2=h_1^2z_1h_2^2z_2[h_2,h_1]^2$ and
$([h_2,h_1]^2)^2=[h_2^2z_2,h_1^2z_1]$. In fact, to prove that any
nilpotent group $G$ of class $2$ such that every element of
$[G,G]$ has a unique square root is the adjoint group of a radical
ring of nilpotency class $3$, Ault and Watters in
\cite{aultwatters} define a sum on $G$ by the rule
$g+h=gh[h,g]^{1/2}$, for $g,h\in G$. In
Proposition~\ref{nilpotent}, every element of $[H,H]$ has a square
root, but it is not necessarily unique.

The counterexample of Hales and Passi to the conjecture of Ault and
Watters is a nilpotent group $G$ of class $2$ such that $G/[G,G]$ is
an indecomposable torsion-free group of rank $2$. It is not known
whether this group $G$ is the adjoint group of a radical ring. It is
an open problem whether every nilpotent group of class $2$ is the
adjoint group of a radical ring.

We note that there exist nilpotent groups of class $2$ which admit a
structure of a left brace that is not a right brace, for
example, the dihedral group of order $8$, see \cite{sysak}.

 It is an open problem whether finite nilpotent groups of higher nilpotency class
admit a structure of left brace (equivalently whether they are IYB
groups). Note that they do not necessarily come from two-sided
braces as  not all such groups are adjoint groups of radical rings
\cite{kruse}.

\begin{example}
{\rm
Let $G=\Sym_{3}$ be the symmetric group of degree $3$. We define an operation
$+$ on $G$ such that $(G,+)$ is a cyclic group generated by $(1,2)$ and
$$2 \cdot (1,2) = (1,2,3),\; 3\cdot  (1,2) = (1,3),\;  4\cdot
(1,2)=(1,3,2),$$ $$5\cdot  (1,2)=(2,3),\; 6\cdot  (1,2)=\id.$$ It is
easy to check that  $(G,+,\cdot)$ is a left brace that is  the
semidirect product of the abelian braces $N$   and $H$ of orders
three and two respectively, where the action of $H$ on $N$ is given
by the inverse map.
 However, $G$ is not a right brace.
Furthermore,
$$\lambda_{(1,2)}=\lambda_{(1,3)}=\lambda_{(2,3)}=((1,2),(2,3))\; ((1,2,3),(1,3,2))\in \Sym_{G}$$
and
$$\lambda_{(1,2,3)}=\lambda_{(1,3,2)}=\lambda_{\id}=\id_{G}.$$
Hence, $\Soc (G) =N$ and the solution of the Yang-Baxter equation
associated to $G$ is a multipermutation solution of level $2$ with
IYB group a cyclic group of order $2$. }
\end{example}

\bibliographystyle{amsplain}

 \vspace{30pt}
 \noindent \begin{tabular}{llllllll}
 F. Ced\'o && E. Jespers  \\
 Departament de Matem\`atiques &&  Department of Mathematics \\
 Universitat Aut\`onoma de Barcelona &&  Vrije Universiteit Brussel  \\
08193 Bellaterra (Barcelona), Spain    &&
Pleinlaan 2, 1050 Brussel, Belgium \\
   &&   \\
J. Okni\'nski &&  \\ Institute of Mathematics &&
\\ Warsaw University&& \\ Banacha 2&& \\ 02-097
Warsaw, Poland &&
\end{tabular}

\end{document}